\documentclass[12pt]{amsart}

\usepackage{pb-diagram}

\usepackage{amsfonts}
\usepackage{amsmath}
\usepackage{amssymb}

\newcommand{\dbar}{\overline{\partial}}

\newcommand{\ddbar}{\sqrt{-1}\partial\dbar}

\newtheorem{theorem}{Theorem}[section]

\newtheorem{lemma}{Lemma}[section]

\newtheorem{definition}{Definition}[section]
\newtheorem{corollary}{Corollary}[section]

\begin{document}

\title{Degeneration of K\"ahler-Einstein manifolds of negative scalar curvature }

\author{Jian Song} 

\address{Department of Mathematics, Rutgers University, Piscataway, NJ 08854}

\email{jiansong@math.rutgers.edu}

\thanks{Research supported in
part by National Science Foundation grants DMS-1406124}

\begin{abstract} Let $\pi: \mathcal{X}^* \rightarrow B^*$ be an algebraic family  of compact K\"ahler manifolds of complex dimension $n$ with negative first Chern class over a punctured disc $B^*\in \mathbb{C}$. Let $g_t$ be the unique K\"ahler-Einstein metric on $\mathcal{X}_t= \pi^{-1}(t)$. We show that  as $t\rightarrow 0$, $(\mathcal{X}_t, g_t)$ converges in pointed Gromov-Hausdorff topology  to a unique finite disjoint union of complete metric length spaces $\coprod_{\alpha=1}^\mathcal{A} (Y_\alpha, d_\alpha)$ with $\sum_{\alpha=1}^\mathcal{A} \textnormal{Vol}(Y_\alpha, d_\alpha) = \textnormal{Vol}(\mathcal{X}_t, g_t)$. Each $(Y_\alpha, d_\alpha)$ is a smooth open K\"ahler-Einstein manifold of complex dimension $n$ outside its closed singular set of Hausdorff dimension no greater than $2n-4$.  Furthermore, $\coprod_{\alpha=1}^\mathcal{A} Y_\alpha$ is a  quasi-projective variety isomorphic to $\mathcal{X}_0 \setminus \textnormal{LCS}(\mathcal{X}_0)$, where $\mathcal{X}_0$ is a projective semi-log canonical model  and $\textnormal{LCS}(\mathcal{X}_0)$ is the non-log terminal locus of $\mathcal{X}_0$. This is the first step of our approach toward compactification of the analytic geometric moduli space of K\"ahler-Einstein manifolds of negative scalar curvature. 
\end{abstract}

\maketitle

{\footnotesize  \tableofcontents}


\section{Introduction}

 The geometric and algebraic moduli space of compact Einstein manifolds is a fundamental problem in both differential geometry and algebraic geometry. Any compact Riemann surface is equipped with a unique metric of constant curvature in its conformal class. Fixing a Riemann surface $M$ with genus greater than one,  the moduli space of complex structures on $M$ admits a unique Deligne-Mumford compactification \cite{DM} by adding Riemann surfaces with possible nodes on the boundary of the moduli. In particular, such Riemann surfaces admit a unique constant curvature metric with complete ends at the nodal points.  In higher dimensions, one considers the Einstein manifolds $(M, g)$ with $$Ric(g) = \lambda g, ~ \lambda= 1, 0, -1$$ instead of spaces of constant curvature. 

In dimension $4$, there are many deep results on the compactness of Einstein manifolds.  In fact, it is proved in \cite{N, An1, T1, An2} that any sequence of $(M_i, g_i)$ with uniform volume lower bound, diameter upper bound and $L^2$-curvature bound will converge to a compact Einstein orbifold, after passing to a subsequence. In the K\"ahler case,  the compactness result holds when $\lambda=0$ with uniform volume upper and lower bounds, and when $\lambda=1$ without any assumption. 

In higher dimension, very little is known about the compactness of the moduli of general Einstein manifolds. In the case of compact K\"ahler-Einstein manifolds, there have been many exciting breakthroughs. The partial $C^0$-estimates, proposed by Tian and established in \cite{T1, DS1},  are the fundamental tool to study analytic and geometric degeneration of K\"ahler-Einstein metrics. Any sequence of K\"ahler-Einstein manifolds $(M_i, g_i)$ with $Ric(g_i) = \lambda g_i$ and $\lambda=1$ must converge in Gromov-Hausdorff topology, after passing to a subsequence, to a singular Kahelr-Einstein metric space homeomorphic to a projective variety with log terminal singularities (\cite{DS1}). Such K\"ahler-Einstein manifolds of positive scalar curvature, the volume must be an integer and the diameter must uniformly bounded above. Therefore, one immediately obtains uniform non-collapsing conditions for all $(M_i, g_i)$ at each point, from the volume comparison theorem. The partial $C^0$-estimate can be then applied to understand the algebraic structures of the limiting K\"ahler-Einstein metric spaces by developing a connection between the analytic and geometric K\"ahler-Einstein metrics and the algebraic Bergman metrics. This approach is also used to prove the fundamental relationship between existence of K\"ahler-Einstein metrics and K-stablity on Fano manifolds (c.f. \cite{CDS1, CDS2, CDS3, T4} ). 

However, when $\lambda=-1$, as one sees in the case of Riemann surfaces of high genus, the Einstein metrics can  collapse at the complete ends. Such complete ends correspond to the nodes from algebraic degeneration of high genus Riemann surfaces. The major result in real dimension $4$ is due to Cheeger-Tian \cite{CT}, much later than the case of positive scalar curvature. They apply the chopping techniques \cite{CG} in the collapsing theory and a refined $\epsilon$-regularity theorem for $4$-folds to establish a non-collapsing result for Einstein $4$-manifolds with uniformly bounded $L^2$-curvature. In particular, let $(M, J_i, g_i)$ be a sequence of K\"ahler-Einstein surfaces  on a smooth manifold $M$ of real dimension $4$, complex structure $J_i$ and $Ric(g_i) = - g_i$. Then it is proved in \cite{CT} that 
$(M, J_i, g_i)$ converge in pointed Gromov-Hausdorff topology, after passing to a subsequence, to a finite disjoint union of complete orbifold K\"ahler-Einstein  surfaces without loss of total volume.

In this paper, we aim to establish the first step towards the compactness of the space of K\"ahler-Einstein manifolds of negative scalar in all dimensions. We will consider any algebraic family 
$$\pi: \mathcal{X}^* \rightarrow B^*$$
over a punctured disc $B^*= B\setminus \{0\}$ in $\mathbb{C}$
such that  $\mathcal{X}_t = \pi^{-1} (t)$ is an $n$-dimensional K\"ahler manifold with $c_1(\mathcal{X}_t)<0$. 
Since $c_1(\mathcal{X}_t)<0$, or equivalently, the canonical bundle of $\mathcal{X}_t$ is ample, there exists a unique K\"ahler-Einstein metric $g_t$  on $\mathcal{X}_t$ satisfying %
\begin{equation}\label{-ke}
Ric(g_t) = - g_t. 
\end{equation}
 A natural question is how $(\mathcal{X}_t, g_t)$ behaves as $t\rightarrow 0$. 
 
 The first result was achieved in \cite{T0}, where Tian considers a very special family $\pi: \mathcal{X} \rightarrow B$ over a disc $B$ in $\mathbb{C}$ satisfying the following 
\begin{enumerate}
\item $\mathcal{X}_t$ is an $n$-dimensional K\"ahler manifold with $c_1(\mathcal{X}_t)<0$ for each $t\in B^*$. 

\medskip

\item The total space $\mathcal{X}$ is smooth. 

\medskip

\item The central fibre $\mathcal{X}_0= \cup_{\alpha=1}^\mathcal{A} X_\alpha$ is reduced and has only normal crossing singularities. In particular, each irreducible component $X_\alpha$ of $\mathcal{X}_0$ is a smooth K\"ahler manifold.

\medskip

\item Any three of the components of $\mathcal{X}_0$ have empty intersection. 

\end{enumerate}
It is shown in \cite{T0} that $(\mathcal{X}_t, g_t)$ must converge smoothly to a disjoint union of complete K\"ahler-Einstein manifolds $\coprod_{\alpha=1}^\mathcal{A} (Y_\alpha, g_\alpha)$ as $t\rightarrow 0$. In particular, $Y_\alpha = X_\alpha \setminus \mathcal{S}$, where $\mathcal{S}$ is the locus of the normal crossing singularities of $\mathcal{X}_0$.  Tian's result was generalized slightly in \cite{LL, Ru1, Ru2}, in particular, condition (4) is removed. However, the assumptions (2) and (3) are too strong. The recent development in algebraic moduli spaces of canonical models shows that in general, the total space has canonical singularities and the central fibre $\mathcal{X}_0$ has semi-log canonical singularities. 

Therefore, we will consider the following degeneration of K\"ahler manifolds of negative first Chern class, i.e., smooth canonical models, in the language of birational geometry. 

\begin{definition} \label{def1.1}

A holomorphic degeneration family $\pi: \mathcal{X} \rightarrow B $ over a disc $B\subset \mathbb{C}$  is said to be a stable degeneration if the following hold.

\begin{enumerate}

\item The total space $\mathcal{X}$ has canonical singularities.

\medskip

\item $\pi$ is a flat projective morphism.

\medskip

\item the fibre $\mathcal{X}_t$ is a smooth canonical model  with $\dim_{\mathbb{C}} \mathcal{X}_t = n$ for each $t\in B^*$, i.e., $c_1(\mathcal{X}_t)<0$. 

\medskip

\item The relative canonical sheaf $K_{\mathcal{X} / B}$ is a $\pi$-ample $\mathbb{Q}$-line bundle. 

\medskip

\item The central fibre $\mathcal{X}_0$ is a semi-log canonical model.

\end{enumerate} 

\end{definition}

The definition of  a semi-log canonical models of general type   and its non-log terminal locus $\textnormal{LCS} $ is given in Section 2 (c.f. Definition \ref{semilog} ). Roughly speaking, a semi-log canonical model of general type   is a projective variety whose codimensional one singularities are ordinary nodes, and its canonical divisor is an ample $\mathbb{Q}$-Cartier divisor. The non-log terminal locus is where the adapted volume measure  is not locally integrable.



Before we state the compactness for $(\mathcal{X}_t, g_t)$ as $t\rightarrow 0$ of a stable degeneration of smooth canonical models of general type, we first construct a canonical K\"ahler-Einstein current on any semi-log canonical models of general type. 

\begin{theorem} \label{main1} Let $X$ be a semi-log canonical model with $\dim_{\mathbb{C}} X = n$. There exists a unique K\"ahler current $\omega_{KE} \in -c_1(X)$ such that 

\begin{enumerate}

\item $\omega_{KE}$ is smooth on $\mathcal{R}_X$, the nonsingular part of $X$,  and it satisfies the K\"ahler-Einstein equation on $\mathcal{R}_X$ 
$$Ric(\omega_{KE} ) = - \omega_{KE}.$$
\item $\omega_{KE}$ has bounded local potentials on the quasi-projective variety $X\setminus \textnormal{LCS}(X)$, where $\textnormal{LCS}(X)$ is the non-log terminal locus of $X$. More precisely, let $$\Phi: X \rightarrow \mathbb{CP}^N$$ be a projective embedding of $X$ by a pluricanonical system $H^0(X, mK_X)$ for some $m\in \mathbb{Z}^+$ and let $\chi = \frac{1}{m} \omega_{FS}|_{X}$, where $\omega_{FS}$ is the Fubini-Study metric on $\mathbb{CP}^N$.  Then  $\omega_{KE} = \chi + \ddbar \varphi_{KE}$ for some $\varphi _{KE}\in PSH(X, \chi)$ such that $$\varphi \in L^\infty_{loc} (X\setminus \textnormal{LCS}(X)), ~~ \varphi \rightarrow - \infty ~\textnormal{near}~ \textnormal{LCS}(X) . $$

\item The Monge-Amp\`ere mass $\omega_{KE}^n$ does not charge mass on the singularities of $X$ and $$\int_{X} \omega_{KE}^n = [K_X]^n. $$

\end{enumerate}

\end{theorem}

The K\"ahler-Einstein currents on semi-log canonical models  are first constructed in \cite{B} using a variational method and they coincide with the K\"ahler-Einstein currents constructed in Theorem \ref{main1} by the uniqueness. The estimates on local boundedness of the K\"ahler-Einstein potentials in \cite{B} are not sufficient to study the Riemannian geometric degeneration  for a stable degeneration of smooth canonical models. Our approach is to combine the fundamental results \cite{Kol1, EGZ, Z} and the maximum principle with suitable barrier functions. This helps us to obtain local $L^\infty$-estimates of local potentials away from the non-log terminal locus of $X$. In fact, it is shown in Section 3 that $\varphi_{KE}$ is milder than any log poles. 

We are now ready to state the following holomorphic compactness for the space of K\"ahler-Einstein manifolds with negative scalar curvature. 

\begin{theorem}\label{main2} Let $\pi: \mathcal{X} \rightarrow B$ be a stable degeneration of smooth canonical models of complex dimension $n$ over a disc $B\subset \mathbb{C}$. Suppose the central fibre $\pi^{-1}(0)$ is given by $\mathcal{X}_0 = \bigcup_{\alpha=1}^{\mathcal{A}} X_\alpha$, where $\{ X_\alpha \}_\alpha$ are the irreducible components of $\mathcal{X}_0$.   Let $g_t$ be the unique K\"ahler-Einstein metric on  $\mathcal{X}_t$ for $t\in B^*$ with 
$$Ric(g_t) = - g_t. $$
Then the following conclusions hold as $t\rightarrow 0$.

\begin{enumerate}

\item There exist points $(p^1_{ t}, p^2_{ t}, ..., p^{\mathcal{A}}_{ t})\in \mathcal{X}_t\times  \mathcal{X}_t \times ... \times  \mathcal{X}_t$ such that $(\mathcal{X}_t, g_t, p^1_t, ..., p^\mathcal{A}_t )$ converge in pointed Gromov-Hausdoff topology to a finite disjoint union of complete K\"ahler-Einstein metric spaces 
$$\mathbf{Y}= \coprod_{\alpha=1}^{\mathcal{A} } (Y_\alpha, d_\alpha, y_\alpha) .$$

\item Let $\mathcal{R}_{Y_\alpha}$ be the regular part of the metric space $(Y_\alpha, d_\alpha)$ for each $\alpha$. Then $(\mathcal{R}_{Y_\alpha}, d_\alpha)$ is a smooth K\"ahler-Einstein manifold of complex dimension $n$ and the singular set $\mathcal{S}_\alpha= Y_\alpha \setminus \mathcal{R}_{Y_\alpha}$ is closed and has Hausdorff dimension no greater than $2n-4$. 

\medskip

\item $\coprod_{\alpha=1}^\mathcal{A} Y_\alpha$ is homeomophic to $\mathcal{X}_0 \setminus \textnormal{LCS}(\mathcal{X}_0)$, where $\textnormal{LCS}(\mathcal{X}_0)$ is the non-log terminal locus of $\mathcal{X}_0$. $\coprod_{\alpha=1}^\mathcal{A}  \mathcal{R}_{Y_\alpha}$ is biholomorphic to the nonsingular part of $\mathcal{X}_0$. 

\medskip
 
\item $\sum_{\alpha=1}^\mathcal{A} \textnormal{Vol}(Y_\alpha, d_\alpha)= \textnormal{Vol}( \mathcal{X}_t, g_t)$ for all $t\in B^*$, where $\textnormal{Vol}(Y_\alpha, d_\alpha)$ is the Hausdorff measure of $(Y_\alpha, d_\alpha)$.

\end{enumerate}
\medskip

\noindent In particular, the K\"ahler-Einstein metric induced by $d_{\mathbf{Y}}$ on $\bigcup_{\alpha=1}^\mathcal{A} \mathcal{R}_{Y_\alpha}$ coincides with the unique K\"ahler-Einstein current on $\mathcal{X}_0$ in Theorem \ref{main1}.

\end{theorem}

Let us say a few words about the proof. In order to prove Theorem \ref{main2}, we first have to achieve a uniform non-collapsing condition for $(\mathcal{X}_t, g_t)$, more precisely, we want to show that there exists $c>0$ and $p_t \in \mathcal{X}_t$  for all $t\in B^*$ such that 
$$Vol_{g_t} (B_{g_t}(p_t, 1) >c, ~ \textnormal{for all}~ t\in B^* . $$
 Without such a condition,  the Cheeger-Colding theory cannot be applied. Such a non-collapsing condition is achieved by uniform analytic estimates using the algebraic semi-stable reduction.
 With the noncollapsing condition, the partial $C^0$-estimates will naturally hold by the fundamental work in \cite{T1, DS1} (also \cite{T3} with an alternative approach). One of the difficult issues in the stable degeneration is that, the diameter will in general tend to $\infty$ and there must be collapsing at the complete ends in the limiting metric space. We will have to estimate the distance near the singular locus of $\mathcal{X}_0$. It turns out that the general principle for geometric K\"ahler currents applies, i.e., boundedness of the local potential is equivalent to boundedness of distance to a fixed regular base point. Such a principle is achieved by building a local Schwarz lemma with suitable auxiliary K\"ahler-Einstein currents and the $L^\infty$-estimates for degenerate complex Monge-Amp\`ere equations from the capacity theory \cite{Kol1, EGZ, Z}.

The recent progress in birational geometry enables the KSBA compatification of canonical models, where the boundary points are semi-log canonical models. Then Theorem \ref{main2} also shows that the K\"ahler-Einstein current on smoothable semi-log canonical models, or boundary points of the moduli of smooth canonical models, are indeed Riemannian geometric with good behaviors near singularities.   
The work of Hacon-Xu \cite{HX} shows that given any algebraic family of smooth canonical models  $\mathcal{X}^*$ over a punctured disc $B^*\subset \mathbb{C}$, after a base change for $B^*$, one can uniquely fill in a central fibre $\mathcal{X}_0$ such that $\mathcal{X}_0$ is a semi-log canonical model and the total space $\mathcal{X}$ has only canonical singularities.   Therefore, by assuming the results in birational geometry, we can remove the assumption on the stable degeneration in Theorem \ref{main2} and obtain a general holomorphic compactness result for K\"ahler-Einstein manifolds of negative scalar curvature.  

\begin{theorem} \label{main3}

Let $\pi: \mathcal{X}^* \rightarrow B^*$ be an algebraic family  of  compact K\"ahler manifolds of complex dimension $n$ over a punctured disc $B^*\subset \mathbb{C}$. Suppose   $c_1(\mathcal{X}_t) <0$ for each  $\mathcal{X}_t= \pi^{-1}(t)$, $t\in B^*$. Let $g_t$ be the unique K\"ahler-Einstein metric on  $\mathcal{X}_t$ with 
$$Ric(g_t) = - g_t.$$
Then the following  holds as $t\rightarrow 0$.

\begin{enumerate}

\item $(\mathcal{X}_t, g_t )$ converges in pointed Gromov-Hausdoff topology to a finite disjoint union of complete K\"ahler-Einstein metric spaces %
$$\mathbf{Y} =\coprod_{\alpha=1}^{\mathcal{A} } (Y_\alpha,  d_{\alpha}) .$$

\item Let $\mathcal{R}_{Y_\alpha}$ be the regular part of $(Y_\alpha, d_\alpha)$ for each $\alpha$. Then $(R_{Y_\alpha}, d_\alpha)$ is a smooth K\"ahler-Einstein manifold of complex dimension $n$ and the singular set $\mathcal{S}_\alpha= Y_\alpha \setminus \mathcal{R}_{Y_\alpha}$ is closed and has Hausdorff dimension no greater than $2n-4$. 

\medskip

\item There exists a unique projective variety $\mathcal{X}_0$ with semi-log canonical singularities and an ample $\mathbb{Q}$-Cartier canonical divisor $K_{\mathcal{X}_0}$ such that  $\coprod_{\alpha=1}^\mathcal{A} Y_\alpha$ is homeomophic to $\mathcal{X}_0 \setminus \textnormal{LCS}(\mathcal{X}_0)$, where $\textnormal{LCS}(\mathcal{X}_0)$ is the non-log terminal locus of $\mathcal{X}_0$. In particular, $\coprod_{\alpha=1}^\mathcal{A}  \mathcal{R}_{Y_\alpha}$ is biholomorphic to the nonsingular part of $\mathcal{X}_0$. 

\medskip
 
\item $\sum_{\alpha=1}^\mathcal{A} \textnormal{Vol}(Y_\alpha, d_\alpha)= \textnormal{Vol}(\mathcal{X}_t, g_t)$ for each $t\in B^*$.

\end{enumerate}

\end{theorem}

We believe that Theorem \ref{main3} will be crucial to establish a general compactness for the space of  K\"ahler-Einstein manifolds with negative scalar curvature. Certainly, at this moment, we will have to assume the algebraic KSBA compatification of the moduli space of canonical models and turn the weak semi-stable reduction over higher dimensional base into suitable analytic and geometric estimates. It is also possible one can directly prove Theorem \ref{main3} without using the algebraic results of \cite{HX}. If so, one might use the compactness for holomorphic families of K\"ahler-Einstein manifolds to construct algebraic fill-in and use such an analytic approach to study the algebraic moduli problems.  

There have been many results for the study of canonical K\"ahler metrics on projective varieties of non-negative Kodaira dimension using the K\"ahler-Ricci flow or deformation of K\"ahler metrics of Einstein type \cite{ST1, ST2, S2, FGS}. Such deformation with both analytic and geometric estimates can also have applications in algebraic geometry, for example, an analytic proof of Kawamata's base point free theorem for minimal models of general type is obtained in \cite{S3}. More generally, the K\"ahler-Ricci flow is closely related to the minimal model program and the formation of finite time singularities is an analytic geometric transition of solitons for birational flips. A detailed program is laid out in \cite{ST3} with partial geometric results obtained in \cite{SW1, SW2, SY, S1}. 

Theorem \ref{main2} is a major improvement of Theorem 1.5 in \cite{S2}. This paper will be incorporated and replace Section 5 in \cite{S2}.

The organization of this article is as follows. A quick review is given in Section 2 for algebraic singularities and degenerations of canonical models.  In Section 3, we prove Theorem \ref{main1} by establishing the existence and uniqueness of canonical K\"ahler-Einstein currents on semi-log canonical models with effective analytic estimates. In Section 4, we derive various analytic estimates for the family of K\"ahler-Einstein metrics with a stable degeneration. In particular, we prove a uniform noncollapsing result. In Section 5, we will apply the Cheeger-Colding theory for the Riemannian degeneration of K\"ahler-Einstein manifolds and establish local partial $C^0$-estimates. In Section 5, we prove the distance estimates and establish the principle that bounded local potential is equivalent to bounded local distance. In Section 6, we combine all the previous results and prove Theorem \ref{main2}.


\section{Semi-log canonical models and stable families of canonical models}

In this section, we will have a quick review of semi-log canonical models and the algebraic degeneration of canonical models of general type. First, let us recall the definition for log canonical singularities of a projective variety. 

\begin{definition} \label{sing} Let $X$ be a normal projective variety such that $K_X$ is a $\mathbb{Q}$-Cartier divisor. Let $ \pi :  Y \rightarrow X$ be a log resolution and $\{E_i\}_{i=1}^p$ the irreducible components of the exceptional locus $Exc(\pi)$ of $\pi$. There there exists a unique collection $a_i\in \mathbb{Q}$ such that
$$K_Y = \pi^* K_X + \sum_{i=1}^{ p } a_i E_i .$$ Then $X$ is said to have
\begin{enumerate}

\item[$\bullet$] terminal singularities if  $a_i >0$, for all $i$.

\medskip

\item[$\bullet$] canonical singularities if $a_i \geq 0$, for all $i $.

\medskip

\item[$\bullet$]  log terminal singularities if $a_i > -1$, for all $i $.

\medskip

\item[$\bullet$]  log canonical singularities if $a_i \geq -1$, for all $i$.

\medskip

\end{enumerate}
A projective normal variety $X$ is said to be a canonical model of general type if $X$ has canonical singularities and $K_X$ is ample. 

\end{definition}

Smooth canonical models are simply K\"ahler manifolds of negative first Chern class. There always exists a unique K\"ahler-Einstein metric on smooth canonical models \cite{A, Y1}.  The notion of semi-log canonical models is crucial  in the degeneration of smooth canonical models (c.f. \cite{K1, Kov}).

\begin{definition} \label{semilog} A reduced projective variety $X$ is said to be a semi-log canonical model  if 

\begin{enumerate}
\item $K_X$ is an ample $\mathbb{Q}$-Cartier divisor, 
\item $X$ has only ordinary nodes in codimension $1$, 

\item For any log resolution $\pi: Y \rightarrow X$, 
$$K_Y = \pi^* K_X  + \sum_{i=1}^I a_i E_i - \sum_{j=1}^J  F_j,$$
where  $E_i$ and $F_j$, the irreducible components of exceptional divisors,  are smooth divisors of  normal crossings with $a_i > -1$.
\end{enumerate}
We denote the algebraic closed set  $\textnormal{LCS}(X) =Supp\left(  \pi\left(\bigcup_{j=1}^J F_j \right) \right)$ in $X$ to be the locus of non-log terminal singularities of $X$. We also let $\mathcal{R}_X$ be the nonsingular part of $X$ and $\mathcal{S}_X$ the singular set of $X$.

\end{definition}
We always require $X$ is $\mathbb{Q}$-Gorenstein and satisfies Serre's $S_2$ condition. We refer such algebraic notions to \cite{Kov}.  The non-log terminal locus is the set of singular points of $X$ with their discrepancy equal to $-1$. Analytically, such locus is the set of points near which the adapted volume measure is locally not integrable.

In general, $X$ is not normal. We first let $\nu: X^\nu \rightarrow X$ be the normalization of $X$ and 
$$K_{X^\nu} = \nu^* K_X - cond(\nu), $$ 
where $cond(\nu)$ is a reduced effective divisor and is called the conductor. In fact, $cond(\nu)$ is the inverse image of codimensional $1$ ordinary nodes. $K_{X^\nu}+ cond(\nu)$ is a big and semi-ample divisor on $X^\nu$ and so the pair $(X^\nu, K_{X^\nu}+cond(\nu))$ has log canonical singularities. 
Let $\pi^\nu: Y \rightarrow X^\nu$ be a log resolution. Then $\pi = \nu \circ \pi^\nu: Y \rightarrow X$ is also a log resolution and 
$$K_Y = \pi^*K_X + \sum a_i E_i  - \sum F_j, ~ a_i >-1, $$
where $E_i$, $F_j$ are the exceptional prime divisors of $\pi$.  In other words, the normalization $(X^\nu)$ of $X$ gives  rise to a log canonical pair $(X^\nu, cond(\nu) )$. In particular,  if one aims to construct a K\"ahler-Einstien metric $g_{KE}$ on $X$, $g_{KE}$ can be pull backed to a suitable K\"ahler current on $X^\nu$ in the class of $K_{X^\nu} + cond(\nu)$.

We now consider the following algebraic degeneration of K\"ahler manifolds with negative first Chern class.

\begin{definition}\label{stabledeg}

A flat projective morphism $\pi: \mathcal{X} \rightarrow B$ over a smooth Riemann surface $B$ is called a stable degeneration of  smooth canonical models  if it satisfies the following conditions.   

\begin{enumerate}

\item $\mathcal{X}$ is a normal projective variety with canonical singularities,

\item $K_{\mathcal{X}}$ is an $\pi$-ample $\mathbb{Q}$-Cartier divisor, 

\item $\mathcal{X}_t = \pi^{-1}(t)$ is a smooth canonical model  of complex dimension $n$, for $t\neq 0$, 

\item $\mathcal{X}_0 = \pi^{-1}(0)$ is a semi-log canonical model. 

\end{enumerate}

\end{definition}

The analysis for K\"ahler-Einstien metrics on singular varieties is usually carried out on the nonsingular model, for example, the log resolution of the original variety. Therefore we would like to have a good nonsingular model for the stable degeneration $\pi: \mathcal{X} \rightarrow B$ so that $\mathcal{X}$ will be nonsingular and the central fibre is a reduced divisor of smooth components with normal crossings. The following resolution of singularities for families over a Riemann surface is called the semi-stable reduction, proved in \cite{KKMS}. 
 
\begin{theorem} \label{reduction} Let $\pi: \mathcal{X} \rightarrow B$ be a morphism from $\mathcal{X}$ onto a Riemann surface $B$. Suppose $0\in B$ and $\pi: \mathcal{X} \setminus \pi^{-1}(0) \rightarrow B\setminus \{0\}$ is smooth. Then there exists a finite base change $f: B' \rightarrow B$ and a blow-up morphism $\Psi: \mathcal{X}' \rightarrow \mathcal{X}\times_B B'$
 \begin{equation}
\begin{diagram}
\node{\mathcal{X}'} \arrow{se,l}{ \pi' }  \arrow{e,t}{\Psi}   \node{\mathcal{X}\times_B B'}  \arrow{e,t}{f'} \arrow{s}  \node{\mathcal{X} } \arrow{s,r}{\pi} \\
\node{}      \node{B'} \arrow{e,t}{f}  \node{B}
\end{diagram}
\end{equation}
 such that the induced morphism $\pi' : \mathcal{X}' \rightarrow B'$ is semi-stable, i.e., 

\begin{enumerate}

\item $\mathcal{X}'$ is nonsingular, 

\medskip

\item $(\pi')^{-1}(0)$ is a reduced divisor with nonsingular components of  normal crossings.

\end{enumerate}

\end{theorem}

In this paper, we will always assume $B$ is a unit disc in $\mathbb{C}$ and let  $t$ be the Euclidean holomorphic coordinate on $B$. The special fibre is given by $\pi^{-1}(0)$. We can calculate the relations of the canonical divisors in the case that $\pi: \mathcal{X} \rightarrow B$ is a stable degeneration of smooth canonical models. We assume that the base change $f$ is given by $t= f(t')= (t')^d$ for some $d\in \mathbb{Z}^+$. 

Let $\Psi': \mathcal{X}' \rightarrow \mathcal{X}$ be the morphism of the semi-stable reduction induced by $\Psi$ and $f$. Since $\mathcal{X}$ has canonical singularities,  $\mathcal{X}'$ also has canonical singularities.  The base change  has ramification along the central fibre of degree $d-1$ and so we have  
$$K_{\mathcal{X}\times_B B'} = (f')^*K_{\mathcal{X}} + (d-1) X'_0,$$
where $X'_0$ is the strict transform of $\mathcal{X}_0$. Therefore

$$K_{\mathcal{X}'} =(\Psi')^* K_\mathcal{X}  + (d-1)\mathcal{X}'_0 + \sum a_j E_j $$
where  $\mathcal{X}'_0 =(\pi')^{-1}(0)$ and $E_j$ are nonsingular components of the exceptional divisor with $a_j \geq 0$. 

The following adapted volume measure is introduced in \cite{EGZ} to study the complex Monge-Amp\`ere equations on singular projective varieties. 

\begin{definition} \label{adapvol} Let $Y$ be a projective variety of normal singularities with a $\mathbb{Q}$-Cartier canonical divisor $K_Y$.  $\Omega$ is said to be an adapted measure on $Y$ if  for any $z\in Y$, there exists an open neighborhood $U$ of $z$ such that $$\Omega = f_U ( \alpha \wedge \overline{\alpha})^{\frac{1}{m}},$$ where $f_U$ is the restriction of a smooth positive function on the ambient space of a projective embedding $U$ and $\alpha$ is a local generator of the Cartier divisor $mK_Y$ on $U$ for some $m\in \mathbb{Z}^+$.
\end{definition}
The adapted volume measure can also be defined for semi-log canonical models via pluricanonical embeddings. However, such volume measures are not locally integrable near the non-log terminal locus. 

\begin{lemma} Let $\pi: \mathcal{X} \rightarrow B$ be a  stable degeneration of smooth canonical models  and $\pi': \mathcal{X}' \rightarrow B'$ be a semi-stable reduction as in Theorem \ref{reduction}. If $\sigma$ is a holomorphic section of $mK_{\mathcal{X}/
B}$  for some $m\in \mathbb{Z}^+$ and $t'$ is holomorphic coordinate of $B'$, then $dt'  \wedge d\overline{t'} \wedge (\Psi')^*\left((\sigma\wedge \overline{\sigma} )^{1/m} \right)$ is a smooth nonnegative real $(2n+2)$-form on $\mathcal{X}'$. 

\end{lemma}

\begin{proof}  By definition of relative canonical bundle,  $dt\wedge d\overline t \wedge (\sigma\wedge \overline{\sigma})^{1/m}$ defines a smooth volume form on the regular part $\mathcal{R}_\mathcal{X}$ of $\mathcal{X}$, and for any adapted volume form $\Omega$ on $\mathcal{X}$ 
$$\sup_{\mathcal{R}_X} \frac{dt\wedge d\overline t \wedge (\sigma\wedge \overline{\sigma})^{1/m}}{\Omega} < \infty.$$
We consider the pullback of the relative volume form $(\Psi')^*\left(\sigma\wedge \overline{\sigma} \right)^{1/m})$. Let $\Omega'$ be a smooth volume form on $\mathcal{X}'$. Then we have
$$\frac{ d t' \wedge d\overline{t'} \wedge (\Psi')^*\left( (\sigma\wedge \overline{\sigma} \right)^{1/m})}{\Omega'} =\left( \frac{ (\Psi')^*\Omega}{d^2|t' |^{2(d-1)} \Omega'}  \right) (\Psi')^*\left( \frac{dt\wedge d\overline t \wedge (\sigma\wedge \overline{\sigma})^{1/m}}{\Omega} \right).$$
Since the volume measure $(\Psi')^*\Omega$ on $\mathcal{X}'$ vanishes along $\mathcal{X}'_0$ of order at least $2(d-1)$, therefore $d t \wedge d\overline{t'} \wedge (\Psi')^*\left((\sigma\wedge \overline{\sigma} )^{1/m} \right)$ is a smooth nonnegative real $2n+2$-form on $\mathcal{X}'$.

\end{proof}

The following theorem is due to Hacon and Xu \cite{HX}. It establishes the unique compactification of an algebraic family  of smooth canonical models  over a punctured disc $B^*$ in $\mathbb{C}$. In the particular, the unique fill-in for the central fibre, possibly after a base change, is a semi-log canonical model. 

\begin{theorem} \label{fillin} Let  $\pi: \mathcal{X}^* \rightarrow B^*$ be an algebraic family  of smooth canonical models over a punctured dicc $B^* = B\setminus \{0\} \subset \mathbb{C}$. Then after a possible base change $f: B' \rightarrow B$, there exists a unique $\pi': \mathcal{X'} \rightarrow B'$ as an extension of $\mathcal{X}^*\times_{B^*} (B')^* \rightarrow (B')^*$ such that $\mathcal{X}'$ has canonical singularities and the central fibre $\mathcal{X}'_0= (\pi')^{-1}(0)$ is a semi-log canonical model. 

\end{theorem}

Theorem \ref{main3} will immediately follow from Theorem \ref{main2} if we apply the algebraic result of Theorem \ref{fillin}. 

\section{K\"ahler-Einstein currents on semi-log canonical models}

In this section, we will construct unique K\"ahler-Einstein currents on semi-log canonical models. Such canonical currents are already constructed in \cite{B} by a variational method. We give a different approach by combining the work of \cite{Kol1, EGZ, Z} and the maximum principle to  establish  $L^\infty$-estimate for the local potentials of such K\"ahler-Einstein currents away from $\textnormal{LCS}$, the non-log terminal locus. Such estimates are particularly important to study the geometric degeneration of smooth canonical models and to establish Riemannian geometric properties of such analytically constructed K\"ahler-Einstein currents.

Let $X$ be a semi-log canonical model of with $\dim_{\mathbb{C}}X= n$. We follow the standard scheme for the construction of K\"ahler-Einstein currents by reducing the K\"ahler-Einstein equation on $X$ to a complex Monge-Amp\`ere equation on the nonsingular model of $X$.  

The first technical preparation is to construct good barrier functions for estimates of local potentials. Such construction relies on a suitable choice of complex hypersurfaces in $X$.

\begin{lemma} \label{effdiv}  For any point $p \in  X\setminus \textnormal{LCS}(X) $, there exists an effective $\mathbb{Q}$-divisor $G_p$ such that $G_p$ is numerically equivalent to $K_X$ and
$$p\notin G_p, ~  \textnormal{LCS}(X)  \subset G_p. $$

\end{lemma}

\begin{proof} Since $X$ is projective, we can choose an effective ample $\mathbb{Q}$-divisor $G_1$ such that $p\notin G_1$. We let $\sigma_1$ be the defining section of $m_1G_1$ for some $m_1\in \mathbb{Z}^+$ so that $m_1G_1$ is Cartier.  Since $G_1$ is ample, $X_1= X\setminus G_1$ is an affine variety in $\mathbb{C}^N$ for some large $N$. Hence we can pick a polynomial $f$ on $\mathbb{C}^N$ such that $f(p)\neq 0$ and $f$ vanishes on the affine variety $\textnormal{LCS} (X) \cap X_1$. The restriction of $f$ to $X$ is a meromorphic function with poles along $G_1$. By choosing sufficiently large $k\in \mathbb{Z}^+$,  $\sigma_2= f \sigma_1^k$ is a holomorphic section of a line bundle associated to the Cartier divisor $G_2= km_1 G_1$ over $X$ such that  $\sigma_2$ vanishes on $\textnormal{LCS}(X)$ and $\sigma_2(p)\neq 0$.   For sufficiently large $m_2\in \mathbb{Z}^+$, $m_2 K_X  - G_2 $ is an ample Cartier divisor. So there exist $m_3\in \mathbb{Z}^+$ and  a holomorphic section $\sigma_3 \in H^0(X, m_3(m_2 K_X- G_2 ))$ such that $\sigma_3(p)\neq 0$. Finally we let $mG_p$ be the divisor of zeros of  $ \sigma_2^{m_3} \sigma_3 $ and the lemma is proved.

\end{proof}

For any $q\in \textnormal{LCS}(X)$,  we can also find $G_{p_1}, ..., G_{p_k}$, by the construction of $G_p$ in Lemma \ref{effdiv}, such that  the common zeros of $G_{p_1}, ..., G_{p_k}$ in an open neighborhood $U$ of $q$ coincide with $\textnormal{LCS}(X)\cap U$.  Without loss of generality, we can consider the projective embedding $\Phi: X \rightarrow \mathbb{CP}^N$ by the pluricanonical system $H^0(X, mK_X) $ for some sufficiently large $m$ defined by $\Phi(x) = [ \eta_0(x), ..., \eta_N(x)] \in \mathbb{CP}^N$ for a basis $\{\eta_j \}_{j=0}^N$ of $H^0(X, mK_X)$. We then choose the adapted volume measure $\Omega$ on $X$ by 
$$\Omega = \left( \sum_{j=0}^N \eta_j \overline{\eta_j} \right)^{1/m}. $$
We also define the curvature of $\Omega$ by
$$\chi = \ddbar \log \Omega.$$
In fact, we can define $h_\Omega = \left( \sum_{j=0}^N \eta_j \wedge \overline{\eta_j} \right)^{-1/m}$ as a hermtian metric on $K_X$.  Obviously, $m\chi$ is the restriction of the Fubini-Study metric of $\mathbb{CP}^N$ to $X$.

The complex Monge-Amp\`ere equation of our interest in relation to the K\"ahler-Einstein equation on $X$ is given by 
\begin{equation} \label{ocmp}
(\chi + \ddbar \varphi)^n = e^\varphi \Omega. 
\end{equation}
We have to prescribe singularities of the solution $\varphi$ to obtain a canonical and unique K\"ahler-Einstein current on $X$. To do so, we   lift all the data to a log resolution $\pi: Y \rightarrow X$. 
By definition of semi-log canonical singularities,
 $$K_Y = \pi^* K_X + \sum_i a_i E_i - \sum_j b_j F_j, ~a_i\geq 0, ~ 0< b_j \leq 1. $$

We approximate  equation (\ref{ocmp})  in  the following way. We pull back all the data from $X$ to $Y$. Let $\sigma_{E}$ be the defining section for $E=\sum_{i=1}^I a_i E_i$ and $\sigma_F$ be the defining section for $F=\sum_{j=1}^J  b_j F_j$ (possibly multivalued). We  equip the line bundles associated to $E$ and $F$ with smooth hermitian metric $h_E$, $h_F$ on $Y$. Let $\Omega_Y$ be a smooth strictly positive volume form on $Y$ defined by 
$$ \Omega_Y  = ( |\sigma_E|^2_{h_E} )^{-1} |\sigma_F|^2_{h_F} \Omega. $$

Let $\theta$ be a fixed smooth K\"ahler form on $Y$ and we consider the following family of complex Monge-Amp\`ere equations on $Y$ for $s \in (0, 1)$, 
\begin{equation}\label{appcmp}
(\chi + s\theta + \ddbar \psi_s)^n = e^{\psi_s} (|\sigma_E|^2_{h_E}+ s)  (|\sigma_F|^2_{h_F}+ s)^{-1}\Omega_Y. 
\end{equation}
For each $0<s<1$, by the result of \cite{A, Y1}, there exists a unique smooth solution $\psi_s$ solving equation (\ref{appcmp}). When $s=0$, equation (\ref{appcmp}) coincides with equation (\ref{ocmp}).

We introduce two more parameter $\delta$ and $\epsilon$ in order to apply the maximum principle and consider the following family of complex Monge-Amp\`ere  equations
\begin{equation}\label{appcmp2}
\left((1+\delta) \chi  + s\theta+ \ddbar\psi_{p, s, \delta, \epsilon} \right)^n = \frac{ e^{\psi_{p, s, \delta, \epsilon}  } (|\sigma_{G_p}|_{h_\Omega}^{2\epsilon} + s) (|\sigma_E|^2_{h_E}+ s)  }{ (|\sigma_F|^2_{h_F}+ s)} \Omega_Y. 
\end{equation}
Here both $\delta$ and $\epsilon$ are sufficiently small and we require $\epsilon>0$. 
\begin{lemma} \label{0logp}
For any $p\in  X \setminus \textnormal{LCS}(X)$ and any  $\epsilon_0>0$, there exist $\delta_0>0$, $C>0$ and $C'=C'(p, \epsilon_0)>0$ such that for any $-\delta_0 \leq \delta \leq \delta_0$, $0<s<1$, and $0< \epsilon <\epsilon_0/2$, the solution $\psi_{p, s, \delta, \epsilon}$ of equation (\ref{appcmp2}) satisfies the following estimate on $Y$, 
\begin{equation}
\epsilon_0 \log |\sigma_{G_p}|_{h_\Omega}^2 - C' \leq  \psi_{p, s, \delta, \epsilon}  \leq   C. 
\end{equation} 
\end{lemma}

\begin{proof}  

We will pick an effective $\mathbb{Q}$-Cartier divisor $G_p$ from Lemma \ref{effdiv}. Since $G_p$ is a $\mathbb{Q}$-Cartier divisor numerically equivalent to the ample divisor $K_X$, we will choose $h_\Omega$ as the hermitian metric for $G_p$ and we let $\sigma_{G_p}$ be the defining section of $G_p$.  
 We first fix a sufficiently small $\delta_0 \geq 3 \epsilon_0>0$ and consider the following family of equations on $Y$
\begin{equation}\label{appp}
\left( (1+\delta) \chi + s\theta + \ddbar \psi_{p, s, \delta, \epsilon_0} \right)^n = \frac{ e^{\psi_{p, s, \delta,\epsilon_0}} (|\sigma_{G_p}|_{h_\Omega}^{2\epsilon_0} + s) (|\sigma_E|^2_{h_E}+ s) }{ |\sigma_F|^2_{h_F}+ s}\Omega_Y,
\end{equation}
where $-\delta_0 \leq \delta \leq \delta_0$. 

Since $G_p$ vanishes along $F$, there exist $\eta=\eta(\epsilon_0)>0$ and $K=K(\epsilon_0, \delta_0)>0$ such that for all $0<s<1$, we have
$$ \left|\left|\frac{ (|\sigma_{G_p}|_{h_\Omega}^{2\epsilon_0} + s) (|\sigma_E|^2_{h_E}+ s) }{ |\sigma_F|^2_{h_F}+ s} \right|\right|_{L^{1+\eta}(Y, \Omega_Y)} \leq K. $$
Then the results of \cite{Kol1, EGZ, Z} imply that there exists $C_1=C_1(p, \delta_0, \epsilon_0)>0$ such that for all $3|\delta|\leq \delta_0$, $0<s<1$, 
$$||\psi_{p, s, \delta, \epsilon_0}||_{L^\infty(Y)} \leq C_1. $$

 Now we will  compare $\psi_{p, s, \delta, \epsilon}$ to $\psi_{p, s, \delta', \epsilon_0}$ by applying the maximum principle. Let $$\phi = \psi_{p, s, \delta, \epsilon} - \psi_{p, s, \delta', \epsilon_0} - \epsilon_0 \log |\sigma_{G_p}|_{h_\Omega}^2. $$  
 Then $\phi$ satisfies the following equation
\begin{equation}\label{appmac}
 \frac{ \left( (1+ \delta')\chi + s\theta + \ddbar \psi_{p, s, \delta', \epsilon_0}  +  (\delta- \delta'-\epsilon_0) \chi  + \ddbar \phi  \right)^n }{ \left( (1+ \delta')\chi + s\theta + \ddbar \psi_{p, s, \delta', \epsilon_0}    \right)^n } = e^{\phi}   \left( \frac{|\sigma_{G_p}|_{h_\Omega}^{2\epsilon}+ s}{1+ s|\sigma_{G_p}|_{h_\Omega}^{-2\epsilon_0}} \right). 
 \end{equation}

We choose $\delta' = -\delta_0$ and require $0<\epsilon<\epsilon_0$.  Since $\phi$ is smooth away from the zeros of $G_p$ and $\phi$ tends to $\infty$ near zeros of $G_p$, we are able to apply the maximum principle to the minimum of $\phi$ and  there exists $C_2>0$ such that
$$\inf_{\widetilde{X}} \phi \geq - C_2, $$ 
Since $\psi_{p, s, \delta', \epsilon_0}$ is bounded, there exists $C_3>0$ such that for all $3\delta \in (-\delta_0, \delta_0)$, $0<s<1$ and $\epsilon \in (0, \epsilon_0/2)$, 
$$ \psi_{p, s, \delta, \epsilon} \geq - C_3 + \epsilon_0  \log |\sigma_{G_p}|_{h_\Omega}^2. $$

It is easier to obtain the upper bound of $\psi_{p, s, \delta, \epsilon}$. We consider the quantity $$\phi' = \psi_{p, s, \delta, \epsilon} - \psi_{p, s, \delta', \epsilon_0} $$
and use similar argument by choosing $\delta' = \delta_0$.

\end{proof}

Next, we will prove  second order estimates with bounds from suitable barrier functions. There exists an effective Cartier divisor $D$ on $Y$ such that for any sufficiently small $s>0$,
\begin{equation} \label{tsuji}
\pi^* K_X - s D
\end{equation}
is ample. In particular, we can assume that the support of $D$ coincides with the support of the exceptional divisor because $K_X$ is ample on $X$. We let $\sigma_D$ be the defining section of $D$ and choose a smooth hermitian metric $h_D$ on the line bundle associated to $D$ such that for any sufficiently small $s>0$, 
\begin{equation}\label{tsuji1}
\chi - s Ric(h_D) >0 .
\end{equation}

\begin{lemma}\label{2ndem} For any $p\in X\setminus \textnormal{LCS}(X)$, there exist $
\lambda$, $\delta_1$, $\epsilon_1>0$ and $C= C(\delta_1, \epsilon_1)>0$ such that for all $-\delta_1< \delta < \delta_1$, $0<\epsilon<\epsilon_1$ and $0<s<1$,  
\begin{equation}
\sup_Y \left(  |\sigma|^2_{h_D} |\sigma_{G_p}|^2_{h_\Omega}) \right)^{2\lambda} \left( \Delta_\theta \psi_{p, s, \delta, \epsilon} \right) \leq C, 
\end{equation}
where $\Delta_\theta$ is the Laplace operator with respect to the K\"ahler metric $\theta$.

\end{lemma}

\begin{proof} 

Let $\omega = (1+\delta) \chi + s\theta + \ddbar \psi_{p, s, \delta, \epsilon}$. Then we consider the quantity $$H= \log tr_\theta (\omega) - A^3\psi_{p, s, \delta, \epsilon} + A^2 \log |\sigma|^2_{h_D} + A \log |\sigma_{G_p}|^2_{h_\Omega}$$
for some sufficiently large $A>0$ to be determined. Straightforward calculations show that there exists $C>0$ such that for all $\delta\in (-\delta_1,  \delta_1)$, $\epsilon\in (0, \epsilon_1)$ and $0<s<1$, 
\begin{eqnarray*}
\Delta_\omega H &\geq &  tr_\omega( (A^3(1+\delta) - A)\chi -  A^2 Ric(h_D)  - (C - A^3s)\theta  -  n A^3\\
& \geq&  Atr_\omega (\theta) - n A^3.
\end{eqnarray*}
We remark that we view $\chi$ as the pullback of the Fubini-Study metric from the projective embedding of $X$ and so the curvature of $\chi$ is uniformly bounded on the projective space. 
Applying the maximum principle, at the minimal point $x_{min}$ of $H$,
$$tr_\omega(\theta)  \leq n A^2. $$ 
Using the geometric mean value inequality combined with the Monge-Amp\`ere equation (\ref{appcmp2}), there exists $C=C(A)>0$ such that 
$$H(x_{min}) \leq C - A^3( \psi_{p, s, \delta, \epsilon} - A^{-1} \log |\sigma_{G_p}|^2_{h_\Omega}) (x_{min}). $$
for sufficiently large $A>0$. By Lemma \ref{0logp}, $ \psi_{p, s, \delta, \epsilon} - A^{-1} \log |\sigma_{G_p}|^2_{h_\Omega}$ is bounded below for all sufficiently small $\delta$ and $0<\epsilon<< A^{-1}$. Therefore for all sufficiently small $\epsilon>0$ and $\delta$,  there exists $C>0$ such that  on $Y$, we have
$$tr_\theta(\omega) \leq  C\left(  |\sigma|^2_{h_D} |\sigma_{G_p}|^2_{h_\Omega}) \right)^{-2\lambda}. $$  
This proves the lemma.

\end{proof}

The following lemma on local high regularity of $\psi_{p, s, \delta, \epsilon}$ is established by the standard linear elliptic theory after applying Lemma \ref{2ndem} and linearizing the complex Monge-Amp\`ere equation (\ref{appp}).

\begin{lemma} 
For any $p\in X\setminus \textnormal{LCS}(X)$, $k\geq 0$ and any compact  $K \subset\subset \mathcal{R}_X \setminus  G_p $ , there exist $\delta_2>0$, $\epsilon_2>0$  and $C=C(p, k, K, \delta_2, \epsilon_2)>0$ such that for any $-\delta_2 \leq \delta \leq \delta_2$, $0< \epsilon \leq \epsilon_2$ and  $0<s<1$ 
$$ || \psi_{p, s, \delta, \epsilon} ||_{C^k(K)} \leq C. $$

\end{lemma}

Before we take $\delta, \epsilon , s \rightarrow 0$, we derive a uniform estimate with respect to variations by the parameters $\delta$, $\epsilon, $ and $t$.

\begin{lemma}  \label{differ}
For any $p\in X\setminus \textnormal{LCS}(X)$, $k\geq 0$, any compact  $K \subset\subset \mathcal{R}_X \setminus  G_p $ , there exist $\delta_3>0$, $\epsilon_3>0$  and $C=C(p, K, \delta_3, \epsilon_3)>0$ such that for any $-\delta_3 \leq \delta \leq \delta_3$, $0< \epsilon \leq \epsilon_3$ and  $0<s<1$, we have
\begin{equation}\label{1derest}
 \left| \frac{\partial \psi_{p, s, \delta, \epsilon}}{\partial \delta} \right|_{L^\infty(K)} + \left| \frac{\partial \psi_{p, s, \delta, \epsilon}}{\partial \epsilon} \right|_{L^\infty(K)}  + \left|  \frac{\partial \psi_{p, s, \delta, \epsilon}}{\partial s} \right|_{L^\infty(K)} \leq C. 
 \end{equation}
 
\end{lemma}

\begin{proof} By the implicit function theorem, the solutions of (\ref{appp}) must be smooth with respect to the parameters $\delta$, $\epsilon$ and $s$. Let $f= \frac{\partial \psi_{p, s, \delta, \epsilon}}{\partial \delta}$. Then $f\in C^\infty(Y)$ and 
$$\Delta_{p, s, \delta, \epsilon} f= - tr_{\omega_{p, s, \delta, \epsilon}}(\chi) + f, $$
where $\Delta_{p, s, \delta, \epsilon}$ is the Laplace operator associated to the metric %
$$\omega= (1+\delta) \chi + s\theta + \ddbar \varphi_{p, s, \delta, \epsilon} . $$ 
The function  $H = f - 10 \psi_{p, s, \delta, \epsilon}+  \log |\sigma_{G_p}|^2_{h_\Omega} $
satisfies the following equation
$$\Delta_{p, s, \delta, \epsilon} H \geq f - 10n= H + 10 \psi_{p, s, \delta, \epsilon} - \log |\sigma_{G_p}|^2_{h_{\Omega}} - 10n. $$
Then for all sufficiently small $\delta$ and $\epsilon>0$, $H$ is uniformly bounded above and so $f$ is uniformly bounded above on any compact subset in $X\setminus G_p$. Estimates for $ \left| \frac{\partial \psi_{p, s, \delta, \epsilon}}{\partial \epsilon} \right|$ and  
$\left|  \frac{\partial \psi_{p, s, \delta, \epsilon}}{\partial s} \right|$ can be achieved similarly. 

\end{proof}

Now we are able to solve equation (\ref{ocmp}). 

\begin{lemma} \label{existsol} For any $p \in X\setminus \textnormal{LCS}(X)$,   there exists  $\varphi \in PSH(X, \chi)$ satisfying the following conditions. 
\begin{enumerate}

\item $\varphi \in L^\infty_{loc} (X \setminus G_p) \cap C^\infty(\mathcal{R}_X\setminus G_p)$ and $\varphi$ solves   equation (\ref{ocmp}), where $G_p$ is defined in Lemma \ref{effdiv}.
 
 \medskip

\item For any $\epsilon>0$ and $p\in X\setminus \textnormal{LCS}(X)$, there exists $C_{p, \epsilon}>0$ such that 
$$\varphi \geq \epsilon \log |\sigma_{G_p}|^2_{h_\Omega} - C_{p, \epsilon}. $$

\item Let $\omega_{KE}= \chi + \ddbar \varphi$. Then $
Ric(\omega_{KE}) = - \omega_{KE} $ on $\mathcal{R}_X$ and extends to the K\"ahler-Einstein equation as currents to $X$. 

\medskip

\item $\varphi= -\infty$ on $\textnormal{LCS}(X)$.

\medskip

\item $\int_{\mathcal{R}_X \setminus G_p} (\omega_{KE})^n =  [K_X]^n.$ 
\medskip

\item

$\varphi = \lim_{s , \delta, \epsilon\rightarrow 0} \psi_{p, s, \delta, \epsilon}$.

\end{enumerate}

\end{lemma}

\begin{proof}  Since for any $p\in X\setminus \textnormal{LCS}(X)$ and any sufficiently small $\delta$ and $\epsilon$, we have uniform estimates for $\psi_{p, s, \delta, \epsilon}$ away from $G_p$,  for any sequence $s_j , \delta_j, \epsilon_j \rightarrow 0$, we can assume $\psi_{p, s_j, \delta_j, \epsilon_j}$ converges, after passing to a subsequence, to some
$$\varphi \in PSH(Y, \chi) \cap C^\infty(Y\setminus G_p).$$
In particular, there exists $C>0$ and for any $\epsilon>0$, there exists $C_\epsilon>0$ such that
$$\epsilon \log |\sigma_{G_p}|^2_{h_\Omega} - C_\epsilon \leq \varphi \leq  C. $$

(1), (2) and (3) can be proved from the above conclusion by passing the estimates of $\psi_{p, s, \delta, \epsilon}$ to the limit $\varphi$. Furthermore, $\varphi$ solves equation (\ref{ocmp}) on $\mathcal{R}_X$.

(4) can be reduced to the following statement: Suppose $\phi$ is a plurisubharmonic function on the unit ball $B \subset \mathbb{C}^n$ such that $$\int_{B} |z_1|^{-2}e^\phi (\sqrt{-1})^ndz_1\wedge d\overline{z_1}\wedge... \wedge dz_n \wedge d\overline{z_n} <\infty, $$ then $\phi$ tends to $-\infty$ near $B \cap \{z_1=0\}$.  
Such a statement is proved by Berndtsson (c.f. Lemma 2.7 in \cite{B}).   $\varphi$ is locally bounded and the fibre of $\pi$ over the log terminal locus is connected. On the other hand, $\varphi$ tends to $-\infty$ near $\pi^{-1}\left(\textnormal{LCS}(X) \right)$ in $Y$. Otherwise there exists a curve $C$ in exceptional divisor and $C$ intersects at least one exceptional divisor with discrepancy $-1$, and so $\varphi$ must be constant on $C$ since it is pluriharmonic with singularities better than any log poles. This leads to contradiction and so $\varphi$ must tend to $-\infty$ near  $\pi^{-1}\left(\textnormal{LCS}(X) \right)$.  Therefore the  function $\varphi$ can uniquely descend to $X\setminus \textnormal{LCS}(X)$.

(5) follows because $\varphi$ is better than log poles along $G_p$ and thus $\ddbar \varphi$ does not charge any mass along the singularities of $X$. More precisely, we let 
$$v_{\epsilon, K} =  \epsilon \log |\sigma_{G_p}|^2_{h_\Omega} + K, ~~E_{\epsilon, K} = \{ \varphi < v_{\epsilon, K} \}. $$
Then $\chi + \ddbar v_{\epsilon, K} = (1- \epsilon) \chi \geq 0$ on $X \setminus G_p$. For any compact set $U \subset \subset X \setminus G_p$ and $\epsilon>0$, one can choose sufficiently large $K$ such that $U \subset E_{\epsilon, K}\subset \subset X \setminus G_p$. By comparison principle, 
$$\int_{\mathcal{R}_X} (\omega_{KE})^n > \int_{E_{\epsilon, K}} (\chi + \ddbar \varphi)^n \geq \int_{E_{\epsilon, K}}(\chi + \ddbar v_{\epsilon, K})^n \geq (1-\epsilon) \int_U \chi^n.$$
By letting $K\rightarrow \infty$ and then $\epsilon \rightarrow 0$, we show that 
$$\int_{\mathcal{R}_X\setminus G_p} (\omega_{KE})^n \geq [K_X]^n. $$   
Similarly, by comparing $\varphi$ to $u_{\epsilon, K} = -\epsilon \log |\sigma_{G_p}|^2_{h_{\Omega}}$, we can also show that 
$$\int_{\mathcal{R}_X\setminus G_p} (\omega_{KE})^n \leq [K_X]^n. $$

(6) can be proved as follows. Suppose $\varphi' \in PSH(Y, \chi) \cap C^\infty(Y\setminus G_p)$ is a sequential limit of another sequence $\psi_{p, s_j, \delta_j, \epsilon_j}$. Then by the estimates in Lemma \ref{differ}, on any compact set $K\subset\subset X \setminus G_p$, there exists $C>0$ such that for sufficiently large $j>0$, 
$$ \sup_K | \psi_{p, s_j, \delta_j, \epsilon_j} - \psi_{p, s_j', \delta_j', \epsilon_j'}| \leq C \left( |\delta_j - \delta_j'| + |\epsilon_j - \epsilon_j'| + |s_j - s_j'| \right). $$
This implies that
 $$\varphi|_K = \varphi'|_K $$
and so $\varphi = \varphi'$ on $Y$ after unique extensions over $G_p$  since both lie in $PSH(Y, \chi)$. The above argument implies that as $s, \delta, \epsilon \rightarrow 0$, the solution $\psi_{p, s, \delta, \epsilon}$ converges to the unique limit $\varphi$. For different $p'\in X\setminus \textnormal{LCS}(X)$, we can let $\epsilon \rightarrow 0$ for equation (\ref{appcmp2}) with $\epsilon=0$ because Lemma \ref{differ} still holds when $\epsilon=0$.

\end{proof}

The solution constructed in Lemma \ref{existsol} coincides with the K\"ahler-Einstein current in \cite{B}, however, the key difference is that we obtain local boundedness for $\varphi$ near the log terminal singular locus of $X$. We will also prove a uniqueness result, which is different from the uniqueness theorem in \cite{B} and will be crucial in Section 4. 

\begin{lemma} \label{uniq} There exists a unique solution $\varphi \in L^\infty_{loc} (X \setminus \textnormal{LCS}(X))\cap C^\infty(\mathcal{R}_{X})$ satisfying
\begin{enumerate}

\item $(\chi + \ddbar \varphi)^n = e^\varphi \Omega$ on $\mathcal{R}_X$, 

\item there exists $p\in X\setminus \textnormal{LCS}(X)$ such that for any $\epsilon>0$, there exist $C>0$ and $C_{p, \epsilon}>0$ with the following estimate
$$\epsilon \log |\sigma_{G_p}|^2_{h_\Omega} - C_{p, \epsilon}  \leq \varphi \leq C , $$
where $G_p$ is an effective divisor as defined in Lemma \ref{effdiv}.

\end{enumerate}
In particular, $\varphi \in PSH(X, \chi)$ satisfies all the conditions in Lemma \ref{existsol}.

\end{lemma}

\begin{proof}
We first prove the uniqueness.  
Let $\varphi$ be the K\"ahler-Einstein potential constructed in Lemma \ref{existsol} as the limit of $\psi_{p, s, \delta, \epsilon}$ ($s , \delta$, $\epsilon\rightarrow0$) for a given $p\in X\setminus \textnormal{LCS}(X)$.
Suppose there exists another $\varphi' $ satisfying the conditions in the lemma and for any $\epsilon>0$, there exist $C_1>0$ and $C_2= C_2(\epsilon)>0$ such that  
$$\epsilon \log |G_{p'}|^2_{h_\Omega} - C_2 \leq \varphi' \leq C_1,$$
for some $p'\in X\setminus \textnormal{LCS}(X)$.
We would like to show that $\varphi' =  \varphi$.

  We let $$G_{p, p'} = \frac{1}{2}{ G_p + G_{p'} }, $$
Then obviously there exist $C_3>0$ and $C_4=C_4(\epsilon)>0$ such that
$$\epsilon \log |G_{p, p'}|^2_{h_\Omega} - C_4 \leq \varphi \leq C_3, ~\epsilon \log |G_{p, p'}|^2_{h_\Omega} - C_4 \leq \varphi' \leq C_3. $$

We consider the quantity 
$$\phi = \psi_{p, s, -\delta, \epsilon} -  \varphi' + \delta^2 \log |G_{p, p'}|^2_{h_\Omega}+\delta^3\log |\sigma_D|^2_{h_D},$$
where $\sigma_D$ and $h_D$ are defined in (\ref{tsuji}) and (\ref{tsuji1}). 
Then $\phi$ satisfies the following equation on the log resolution $Y$, 
$$\frac{(\chi + \ddbar \varphi'  + s\theta - \delta(1- \delta )\chi + \delta^3 Ric(h_D)+   \ddbar \phi  )^n }{ (\chi + \ddbar \varphi')^n} = e^{\phi}   \frac{ |G_{p'}|^{2\epsilon}_{h_\Omega} + s}{ |G_{p, p'}|_{h_\Omega}^{2\delta^2}  |\sigma_D|^{2\delta^3}_{h_D}   }.$$
We pick $s<<\delta<<1$,  $ \epsilon << \delta^2$ and apply the maximum principle to $\phi$. There exists $C>0$ such that for all $s<< \delta<<1$, $\epsilon << \delta^2$, 
$$\sup_X \phi \leq C. $$
Let $s, \delta, \epsilon \rightarrow 0$. We have
$$\varphi \leq \varphi'.$$
Similarly, we can prove $\varphi \geq \varphi'$ by applying the maximum principle to $$\phi' = \psi_{p', s, \delta, \epsilon} - \varphi' - \delta \log |\sigma_{G_p}|^2_{h_\Omega}. $$
Therefore $\varphi = \varphi'$ on $X \setminus G_{p, p'}$ and $\varphi'$ extends uniquely in $PSH(X, \chi)$ since $\varphi \in PSH(X, \chi)$.

For the existence part, we only need to verify that $\varphi\in L^\infty_{loc}(X \setminus \textnormal{LCS}(X))$. By the uniqueness result we proved above, $\psi_{p, s, \delta, \epsilon}$ converges to the same $\varphi$ and such $\varphi$ satisfies the estimate (2) in Lemma \ref{existsol} for all $p\in X\circ$. Then (2) in the lemma is proved.

\end{proof}

Lemma \ref{uniq} completes the proof of Theorem \ref{main1}. We also remark that the uniqueness result of Lemma \ref{uniq} does not require $\varphi \in PSH(X, \chi)$.  We can also generalize Lemma \ref{uniq} by adding a weight on the righthand side equation.

\begin{lemma} \label{Finfty} For any open set $U \subset\subset X\setminus \textnormal{LCS}(X) $ and fixed $\alpha>1$, $K_1, K_2>0$,  if   $f$ is a real valued function  on $X$ satisfying 
$$||e^f||_{L^\infty(X\setminus U)} \leq  K_1 , ~~ \int_U e^{\alpha f} \Omega \leq K_2, $$
then there exists a unique solution $\varphi$ solving 
$$(\chi+ \ddbar \varphi)^n = e^{\varphi + f} \Omega$$
satisfying  
\begin{enumerate}

\item $\left|\left|\varphi \right|\right|_{L^\infty(U)}  \leq C_1 = C_1(X, U, \Omega, \alpha, K_1, K_2).$

\medskip

\item For any $p\in X\setminus \textnormal{LCS}(X)$ and $\epsilon>0$, there exist $C_2= C_2(X, U, \Omega, \alpha, K_1, K_2)>0$ and $C_3 = C_3 (X, U, \Omega, \alpha, K_1, K_2, \epsilon, p)>0$ such that 
$$- C_3 + \epsilon \log |G_p |^2_{h_\Omega} \leq \varphi \leq C_2. $$

\end{enumerate}
In particular, if the pullback of $f$ on the log resolution $Y$ of $X$ is smooth,  $\varphi \in C^\infty(\mathcal{R}_X)$. 
\end{lemma}

The proof of Lemma \ref{Finfty} is almost identical to the proof of Lemma \ref{existsol} and Lemma \ref{uniq}. It suffices to lift the equation onto the nonsingular model $Y$ of $X$ and approximate $f$ by smooth functions with the same bounds for $f$.  Lemma \ref{Finfty} is a key ingredient to derive distance estimates in Section 6.


\section{Analytic estimates for stable families of smooth canonical models }

In this section, we will derive uniform estimates for K\"ahler-Einstein metrics in  a stable degeneration of smooth canonical models. 

We will use the notations in Section 2. Let $\pi: \mathcal{X} \rightarrow B$ be a stable degeneration of smooth canonical models  over a disc $B\subset\mathbb{C}$. Suppose $\mathcal{X}_t = \pi^{-1}(t)$ is a smooth canonical model of  complex dimension $n$ for $t\in B^*$ and the central fibre $\mathcal{X}_0$ is a semi-log canonical model, where $t$ is the Euclidean holomorphic coordinate on $B$.  After embedding $\mathcal{X}$ into $\mathbb{CP}^{N}$ by pluricanonical sections $\eta_0, ..., \eta_{N} $ in $|m K_{\mathcal{X}/B}|$ for some sufficiently large $ m \in \mathbb{Z}^+$, we let $\chi\in -c_1(K_{\mathcal{X}/B})$ be a smooth K\"ahler form on $\mathcal{X}$ induced from the projecting embedding, i.e., 
$$\chi =\frac{1}{m}  \ddbar \log \left( \sum_{j=0}^{N} |\eta_j|^2 \right).$$ 
Then $\chi_t = \chi|_{\mathcal{X}_t} \in -c_1(\mathcal{X}_t)$ is a smooth K\"ahler form on $\mathcal{X}_t$  for $t\in B^*$. $\chi_0= \chi|_{\mathcal{X}_0}$ is a K\"ahler current on $\mathcal{X}_0$ with bounded local potential and it is smooth on $\mathcal{R}_{\mathcal{X}_0}$, the nonsingular part of $\mathcal{X}_0$. We can define a real valued $(n,n)$-form measure $\Omega$ on $\mathcal{X}$ by
$$ \Omega = \left( \sum_{j=0}^N \eta_j \wedge \overline \eta_j \right)^{1/m} . $$

The restriction $\Omega_t = \Omega|_{\mathcal{X}_t}$ on each general fibre is a smooth non-degenerate volume form on $\mathcal{X}_t$ for all $t\in B^*$. $\Omega_0= \Omega| _{\mathcal{X}_0}$ is an adapted volume measure on $\mathcal{X}_0$. In particular, 
$$\ddbar\log \Omega_t = \chi_t. $$
 We now consider the following family of complex Monge-Amp\`ere equations on $\mathcal{X}_t$ for each $t\in B^*$
 \begin{equation} \label{famke}
 (\chi_t + \ddbar\varphi_t)^n = e^{\varphi_t} \Omega_t. 
 \end{equation}
Equation (\ref{famke}) admits a unique smooth solution $\varphi_t$ for all $t \in B^*$ and $\omega_t = \chi_t + \ddbar \varphi_t$ is the unique smooth K\"ahler-Einstein metric on $\mathcal{X}_t$. We let $g_t$ be the  K\"ahler-Einstein metric associated with $\omega_t$.
 
We will first  compare $\chi^n$ and $\Omega$ on each $\mathcal{X}_t$. 
 
\begin{lemma} \label{vol1} There exists $C>0$ such that 
$$\sup_{\mathcal{X}} \frac{ dt \wedge d\overline t \wedge \chi^n } { dt\wedge d\overline t \wedge \Omega} \leq C .$$

\end{lemma}

\begin{proof} We use a  trick similarly in \cite{EGZ}. By the choice of $\{ \eta_j \}_{j=0}^{N}$, $$\sqrt{-1}dt\wedge d\overline t \wedge \Omega$$ is an adapted volume measure on $\mathcal{X}$.  For any point $p \in \mathcal{X}$, we can embed an open neighborhood $U$ of $p$ in $\mathcal{X}$  into $\mathbb{C}^N$ by $i: U \rightarrow \mathbb{C}^N$ (for example, we can  localize the embedding by $[\eta_0, ..., \eta_{N}])$. Then $\chi|_U $ extends to a smooth K\"ahler metric on $\mathbb{C}^N$ and is quasi-equivalent to the Euclidean metric $\hat\omega = \sqrt{-1} \sum_{i=1}^N dz_i \wedge d\overline z_i$. Hence there exists $C_1>0$ such that near $i(p)$
$$ (C_1 )^{-1} \chi^{n+1} \leq  \sum_{ 1\leq i_1 < i_2 < ...< i_{n+1} \leq N } (\sqrt{-1})^{n+1}\prod_{k=1}^{n+1} dz_{i_k}  \wedge d \overline z_{i_k} \leq C_1 \chi^{n+1}.$$
Since $\sqrt{-1} dt\wedge d\overline t \wedge \Omega$ is an adapted volume measure on $\mathcal{X}$, for any $1\leq i_1 \leq i_2\leq ... \leq i_{n+1}\leq N$, there exist smooth nonnegative functions $f_{i_1, ..., i_{n+1}}$ in $U$  such that
$$ i^* \left(  (\sqrt{-1})^{n+1}\prod_{k=1}^{n+1} dz_{i_k} \wedge d \overline{z}_{i_k} \right)  =  \sqrt{-1} f_{i_1, ..., i_{n+1}} dt\wedge d \overline t\wedge \Omega .$$
%
%
This implies that there exists $C_2>0$ such that 
$$ \chi^{n+1} \leq C_2 ~ \sqrt{-1}dt\wedge d\overline t \wedge \Omega. $$
The lemma is proved since $\sqrt{-1} dt\wedge d\overline t$ is bounded above by a multiple of $\chi$.

\end{proof}

We immediately can achieve the following uniform upper bound for the potential $\varphi_t$. 

\begin{corollary} \label{5upb}There exists $C>0$ such that for all $t\in B^*$, 
$$\sup_{\mathcal{X}_t} \varphi_t \leq C. $$
\end{corollary}

\begin{proof} We apply the maximum principle to equation (\ref{famke}) on $\mathcal{X}_t$ at the maximal point of $\varphi_t$ and obtain
$$ \sup_{\mathcal{X}_t} \varphi_t \leq \sup_{\mathcal{X}_t} \log\left(  \frac{\chi_t^n }{\Omega_t} \right) = \sup_{\mathcal{X}_t} \log \left( \frac{\sqrt{-1} dt\wedge d\overline t \wedge \chi_t^n }{ \sqrt{-1} dt\wedge d\overline t \wedge \Omega_t} \right)\leq \sup_{\mathcal{X}} \log \left( \frac{ \sqrt{-1} dt\wedge d\overline t \wedge \chi^n }{\sqrt{-1} dt\wedge d\overline t \wedge \Omega} \right) . $$
The corollary easily follows from Lemma \ref{vol1}.

\end{proof}

The main goal of this section is to achieve a local $L^\infty$-estimate for $\varphi_t$. We will apply the semi-stable reduction for $\pi: \mathcal{X} \rightarrow B$ with following diagram 
$$
\begin{diagram}
\node{\mathcal{X}'} \arrow{se,l}{ \pi' }  \arrow{e,t}{\Psi}   \node{\mathcal{X}\times_B B'}  \arrow{e,t}{f'} \arrow{s}  \node{\mathcal{X} } \arrow{s,r}{\pi} \\
\node{}      \node{B'} \arrow{e,t}{f}  \node{B}
\end{diagram}
$$
such that  $\mathcal{X}'$ is smooth  and $\mathcal{X}'_0= (\pi')^{-1}(0)$ is a reduced divisor of smooth components with  normal crossings.  Let $t'$ be the holomorphic coordinate on $B'$ such that $t = (t')^d$.

The central fibre $\mathcal{X}'_0 = \widetilde{\mathcal{X}_0} \cup \mathcal{E} $, where  $\widetilde{\mathcal{ X}_0}$ is the strict transform of $\mathcal{X}_0$ and $\mathcal{E}$ is the exceptional divisor.  


For simplicity, we will use  $\chi$ for  $(\Psi')^* \chi$ on $\mathcal{X}'$, where 
$$\Psi' = f' \circ \Psi.$$ By Kodaira's lemma, there exist an effective $\mathbb{Q}$-Cartier divisor $\mathcal{D}$ whose support coincides with the support of $\mathcal{E}$ and a smooth hermitian metric $h_{\mathcal{D}}$ equipped on the line bundle associated to $\mathcal{D}$ such that $  \chi_\epsilon = \chi - \epsilon Ric(h_{\mathcal{D}})$ is a K\"ahler form on $\mathcal{X}'$ for all sufficiently small $\epsilon >0$. Let $\sigma_{\mathcal{D}}$ be the defining section of $\mathcal{D}$.

By discussion in Section 2, we have
$$K_{\mathcal{X}'} = (\Psi')^*K_{\mathcal{X}} + (d-1)\widetilde{\mathcal{X}_0} + \sum_{i=1}^K c_i D_i, $$
with $c_i \geq 0$, where $D_i$ are components of the exceptional divisor $\mathcal{D}$. Also 
Let $\widetilde{\mathcal{X}_0} = \cup_{\alpha=1}^{\mathcal{A}} \widetilde{X_\alpha}$, where $\widetilde{X_\alpha}$ is the irreducible component of $\mathcal{X}'_0$. 

We pick arbitrary $\alpha$ and let 
$$\widetilde{X} = \widetilde{X}_\alpha, ~~ \widetilde{\mathcal{X}_0} = \widetilde{X} \cup \widetilde{X}' $$ and work on $\widetilde{X}$ instead of $\widetilde{\mathcal{X}_0}$ unless $\widetilde{\mathcal{X}_0}$ is already irreducible.  In particular, there exists a component $X$ in $\mathcal{X}_0$ such that 
$$X = \Phi' (\widetilde{X}). $$
We can write
$$\widetilde{\mathcal{X}'_0} = \widetilde X  \cup \widetilde{X}'  \cup_{i=1}^I E_i \cup_{j=1}^J  F_j$$
from the bundle formula 
$$K_{\widetilde{X}} = (\Psi')^* K_X + \sum a_i E_i|_{\widetilde{X}} - \sum_j b_j F_j|_{\widetilde{X}}, $$
%
where $a_i \geq 0$, $0< b_j \leq 1$, $E_i$ and $F_j$ are smooth irreducible components in $(\mathcal{D}\cup \widetilde{X}')$. 
Let $\sigma_{E_i}$, $\sigma_{F_j}$ and $\sigma_{\widetilde{X}'}$ be the defining sections for $E_i$, $F_j$ and $\widetilde{X}'$ respectively and we equip the corresponding line bundles with smooth hermitian metrics $h_{E_i}$, $h_{F_j}$ and $h_{\widetilde{X}'}$.


We will now estimate the lower bound of $\varphi_t$. Equation (\ref{famke}) can be lifted to $\mathcal{X}'_{t'} = (\Psi')^{-1}(t')$ for $t' \in (B')^*$,  given by 
\begin{equation}
(\chi_{t'} + \ddbar \varphi_{t'})^n = e^{\varphi_{t'}} \Omega_{t'}, 
\end{equation}
where $\chi_{t'}$, $\Omega_{t'}$ and $\varphi_{t'}$ are pullback of $\chi_t$, $\Omega_t$ and $\varphi_t$ by $\Psi'$.  

\begin{lemma} \label{54} For any $\epsilon>0$, there exists $C_\epsilon>0$ such that for all $t' \in (B')^*$, we have 
$$ \varphi_{t'} \geq \epsilon \log \left( \prod_{i=1, j =1}^{I , J}   |\sigma_{ E_i} |_{h_{E_i}}^2  |\sigma_{F_j}|^2_{h_{F_j}}  |\sigma_{\widetilde{X}'}|^2_{h_{\widetilde{X}'}} \right)- C_\epsilon$$
on $ \mathcal{X}_t$.

\end{lemma}  

\begin{proof} 

 We define the following smooth closed $(1,1)$-form
\begin{eqnarray*}
  \chi_{t', \epsilon}  
&=&   \chi_{t'} - \epsilon_1 \ddbar \log Ric(h_\mathcal{D}) - \epsilon_2 \ddbar \sum_k \log \left(  |\sigma_{\widetilde{X}'}|^2_{h_{\widetilde{X}'}} \right) \\
&&+  \epsilon_3    \ddbar \left(  \sum_{i=1}^I |\sigma_{E_i}|^{2\epsilon_4}_{h_{E_i}} + \sum_{j=1}^J |\sigma_{F_j}|^{2\epsilon_4}_{h_{F_j}} + |\sigma_{\widetilde{X}}|^{2\epsilon_4}_{h_{\widetilde{X}}} \right) , 
\end{eqnarray*}
where $\epsilon=(\epsilon_1, \epsilon_2, \epsilon_3, \epsilon_4)$. 

For any sufficiently small $\epsilon_1>0$ and $\epsilon_4>0$, we can choose $0<\epsilon_3<< \epsilon_2 << \epsilon_1$ such that  
$\chi_{t', \epsilon}$ is a K\"ahler metric with conical singularities along smooth divisors $E_i$, $F_j$ and $\widetilde{X}'$ with  normal crossings of the same cone angle $2\pi \epsilon_4$.  

Let 
$$   \varphi_{t', \epsilon} =    \varphi_{t'} - f_\epsilon, $$
where 
$$f_\epsilon = \epsilon_1 \log |\sigma_{\mathcal{D}}|^2_{h_{\mathcal{D}}} + \epsilon_2   \log |\sigma_{\widetilde{X}'}|^2_{h_{\widetilde{X}'}}+ \epsilon_3     \left(  \sum_{i=1}^{I} |\sigma_{E_i}|^{2\epsilon_4}_{h_{E_i}} + \sum_{j=1}^{J} |\sigma_{F_j}|^{2\epsilon_4}_{h_{F_j}} + |\sigma_{\widetilde{X}'}|^{2\epsilon_4}_{h_{\widetilde{X}'}} \right). $$ 
Then we have on $\mathcal{X}'_{t'} $ for $t' \in (B')^*$, 
$$
 (   \chi_{t', \epsilon} + \ddbar   \varphi_{t', \epsilon} )^n =e^{    \varphi_{t', \epsilon} + f_\epsilon }   \Omega_{t'}.
$$
We can now apply the maximum principle to $   \varphi_{t', \epsilon}$ on each $\mathcal{X}'_{t'}$. Suppose $   \varphi_{t', \epsilon}$ achieves its minimum at $p_{t', \epsilon}$ in $\mathcal{X}'_{t'}$ for any given $t' \in (B')^*$. Then 
\begin{eqnarray*}
\inf_{\mathcal{X}'_{t'} } \varphi_{t', \epsilon}  
&\geq & \inf_{\mathcal{X}'_{t'}}  \log \left( \frac{ ( \chi_{t', \epsilon})^n}{e^{f_\epsilon} \Omega_{t'}} \right)
\\
&= & \inf_{\mathcal{X}'_{t'}} \log \left(  \frac{ dt' \wedge d\overline{t'} \wedge ( \chi_{t', \epsilon} )^n}{e^{f_\epsilon} dt' \wedge d\overline{t'} \wedge \Omega_{t'}}     \right)
\\
&\geq& \inf_{\mathcal{X}'} \log \left(  \frac{ dt' \wedge d\overline{t'} \wedge ( \chi_{t', \epsilon} )^n}{e^{f_\epsilon} dt' \wedge d\overline{t'} \wedge \Omega_{t'}}     \right)
\\
\end{eqnarray*}
By the normal crossings of the smooth components in $\mathcal{X}'_0$ from the semi-stable reduction, we can cover $\mathcal{X}'_0$ by finitely many small coordinate charts $\{ U_\beta \}_{\beta=1}^\mathcal{B}$ in $\mathcal{X}'$. We can assume in $U_\beta$,  
$$t' = x z_1z_2... z_m, ~ \textnormal{for~ some} ~ m \leq n, $$
where $x$, $z_1$, ..., $z_n$ are local holomorphic coordinates. 
%
 %

We may assume that for some $\beta$, 
$$ \inf_{\mathcal{X}'} \log \left(  \frac{ dt' \wedge d\overline{t'} \wedge \chi_{t', \epsilon}^n}{e^{f_\epsilon} dt' \wedge d\overline{t'} \wedge \Omega_{t'}}     \right)
= \inf_{U_\beta}  \left(  \frac{ dt' \wedge d\overline{t'} \wedge \chi_{t', \epsilon}^n}{e^{f_\epsilon} dt' \wedge d\overline{t'} \wedge \Omega_{t'}}     \right) .$$

\begin{enumerate}

\item Suppose $\widetilde{X} \cap U_\beta \neq \phi$. We then assume that $\widetilde{X} = \{ x=0\}$ in $U_\beta$ and so $\{z_k=0\}$ corresponds to one of $E_i$, $F_j$ and $\widetilde{X}'$.  Let
 $$\omega_{\beta, \epsilon_4}  = \sqrt{-1} \left( dx\wedge d\overline x +  \sum_{k=1}^m |z_k |^{-2+2\epsilon_4} dz_k \wedge d\overline{z}_k+  \sum_{k=m+1}^n dz_k \wedge d\overline{z}_k \right)$$
 be the flat conical K\"ahler metric on $U_\beta$. For each suitable sufficiently small $\epsilon_4$, there exists $C_1=C_1 >0$, 
$$ C_1^{-1}  \omega_{\beta, \epsilon_4} \leq \chi_{t', \epsilon} \leq C_1 \omega_{\beta, \epsilon_4}, $$
where $\epsilon= (\epsilon_1, \epsilon_2, \epsilon_3, \epsilon_4)$.
Since
$$  dt'\wedge d \overline{ t'} = \left(z_1...z_m dx + \sum_{k=1}^m \frac{x z_1... z_m} {z_k} dz_k \right)\wedge \left(\overline z_1... \overline z_m d\overline x + \sum_{k=1}^m \frac{\overline x \overline z_1... \overline z_m}{\overline z_k } d \overline z_k \right), $$
there exists $C_2= C_2 >0$ such that 
$$\sqrt{-1} dt' \wedge  d\overline{t'} \wedge (\chi_{t', \epsilon} )^n\geq (\sqrt{-1})^{n+1} C_2 |z_1z_2...z_m|^{2\epsilon_4}     dx\wedge d\overline x \wedge dz_1\wedge d\overline z_1 \wedge ...\wedge dz_n \wedge d\overline z_n. $$
On the other hand, by the observation in Section 2, $\sqrt{-1} dt' \wedge d \overline{t'} \wedge \Omega_{t'}$ is a smooth nonnegative real valued $(2n+2)$-form on $\mathcal{X}'$ and so there exists $C_3>0$ such that 
$$(\sqrt{-1})^{n+1} dt' \wedge d \overline{t'} \wedge\Omega_{t'}  \leq C_3  (\sqrt{-1})^{n+1}dx\wedge d\overline x \wedge dz_1\wedge d\overline z_1 \wedge ...\wedge dz_n \wedge d\overline z_n. $$
By combining the above estimates, we have
$$\inf_{U_\beta}  \left(  \frac{ dt' \wedge d\overline{t'} \wedge (\chi_{t', \epsilon})^n}{e^{f_\epsilon} dt' \wedge d\overline{t'} \wedge \Omega_{t'}}     \right) \geq  C_3 \inf_{U_\beta} |z_1... z_m|^{2\epsilon_4- 2\epsilon_2}, $$
since $e^{f_\epsilon}$ vanishes along each $z_i$ of order at least $2\epsilon_2$, $i=1, ..., n$. 
As we always choose $0< \epsilon_4 << \epsilon_2$, there exists $C_4>0$, such that
$$ \inf_{U_\beta} \log \left(  \frac{ dt' \wedge d\overline{t'} \wedge ( \chi_{t', \epsilon} )^n}{e^{f_\epsilon} dt' \wedge d\overline{t'} \wedge \Omega_{t'}}     \right) \geq  -  C_4. $$

\item Suppose $\widetilde{X}\cap U_\beta =\phi$. We let
 $$\omega_{\beta, \epsilon_4}  = \sqrt{-1} \left( |x|^{-2+ 2\epsilon_4} dx\wedge d\overline x +   \sum_{k=1}^m  |z_k |^{-2+2\epsilon_4} dz_k \wedge d\overline{z}_k +  \sum_{k=m+1}^n  dz_k\wedge d\overline{z}_k \right) $$
 be the flat conical K\"ahler metric on $U_\beta$. For each suitable sufficiently small $\epsilon$, there exists $C_5=C_5(\epsilon)>0$, 
$$ C_5^{-1}~  \omega_{\beta, \epsilon_4} \leq \chi_{t', \epsilon} \leq C_5 ~ \omega_{\beta, \epsilon_4}  . $$
%
%
Then there exist $C_6 = C_6(\epsilon)>0$ and $C_7=C_7(\epsilon)>0$ such that
$$ \inf_{U_\beta} \left(  \frac{ dt' \wedge d\overline{t'} \wedge ( \chi_{t', \epsilon} )^n}{e^{f_\epsilon} dt' \wedge d\overline{t'} \wedge \Omega_{t'}}     \right) \geq   C_6  \inf_{U_\beta} \left(  |x z_1z_2... z_m|^{-2\epsilon_2 + 2\epsilon_4}  \right) \geq C_7,$$
 for $\epsilon_4<<\epsilon_2$.
\medskip

\end{enumerate}
Combining the above estimates, there exists $C_8=C_8(\epsilon) >0$ such that for all $t'\in (B')^*$, we have 
$$\inf_{\mathcal{X}'_t} \varphi_{t', \epsilon} \geq - C_8. $$ 
The lemma follows immediately from the relation between $\varphi_{t'}$ and $\varphi_{t', \epsilon}$. 

\end{proof}

Since we can obtain estimates for $\varphi_{t'}$ near any component of $\mathcal{X}'_0$, Lemma \ref{54} immediately implies the following local estimates for $\varphi_t$ on $\mathcal{X}$.  

\begin{corollary} \label{lbdf} Let $\mathcal{S}_{\mathcal{X}_0}$ be the singular set of $\mathcal{X}_0$. Then for any compact subset $K\subset \subset \mathcal{X} \setminus \mathcal{S}_{\mathcal{X}_0}$, there exists $C_K>0$ such that 
\begin{equation}
\inf_{t\in B^*} \inf_{K\cap \mathcal{X}_t} \varphi_t \geq - C_K.  
\end{equation}

\end{corollary}

We will need several versions of the Schwarz lemma and the following lemma is the first among them. 

\begin{lemma} \label{sch1} We define the barrier function $F_{\widetilde{X}}= \prod_{i=1, j=1}^{I,J}  |\sigma_{ E_i} |_{h_{E_i}}^2  |\sigma_{F_j}|^2_{h_{F_j}}  |\sigma_{\widetilde{X}'}|^2_{h_{\widetilde{X}'}}$ on $\mathcal{X'}$, given the component $\widetilde X$ in $\mathcal{X}'_0$.   Let $\omega_{t'} = \chi_{t'}+ \ddbar \varphi_{t'}$ be the K\"ahler-Einstein metric on $\mathcal{X}'_{t'}$ for $t'\in (B')^*$. Then for any $\epsilon>0$, there exists $C_\epsilon>0$ such that for all $t ' \in (B')^*$, we have on $\mathcal{X}'_{t'}$ 
\begin{equation}
\omega_{t'} \geq  C_\epsilon  \left(\left. F_{\widetilde{X}} \right|_{\mathcal{X}'_{t'}} \right)^{\epsilon} \chi_{t'} . 
\end{equation}

\end{lemma}

\begin{proof} Let 
$$H = \log \left(\left. F_{\widetilde{X}} \right|_{\mathcal{X}'_{t'}} \right)^{\epsilon} tr_{\omega_{t'}} (\chi_{t'})  - 2A \varphi_{t'} $$ 
be the smooth function on $\mathcal{X}'_{t'}$ for $t' \in (B')^*$. Since $\chi_{t'}$ is the pullback of the Fubini-Study metric from a projective  embedding of $\mathcal{X}$, its curvature is  uniformly bounded. Straightforward calculations show that 
$$\Delta_{t'} H \geq A tr_{\omega_{t'}}(\chi_{t'}) - C_1, $$
for sufficiently large $A>0$ and some uniform constant $C_1>0$ dependent on $A$. Applying the maximum principle and the estimate $\varphi_{t'}$ from Lemma \ref{54}, there exists $C_2>0$ such that for all $t' \in (B')^*$, $$\sup_{\mathcal{X}'_{t'}} H \leq C_2.$$
The lemma follows immediately from the uniform upper bound for $\varphi_{t'}$ in Corollary \ref{5upb}.

\end{proof}

\begin{lemma} \label{55} For any $k>0$ and any compact set $K\subset\subset \mathcal{X} \setminus \mathcal{S}_{\mathcal{X}_0}$, there exists $C_{k,K}>0$ such that for all $t\in B^*$, 
$$||\varphi_t||_{C^k(K\cap \mathcal{X}_t, \chi_t)} \leq C_{k,K}. $$

\end{lemma}
\begin{proof}  By Lemma \ref{sch1} and the original complex Monge-Amp\`ere equation (\ref{famke}) for $\varphi_t$, $\omega_t$ is uniformly bounded above and below with respect to $\chi_t$ away from the singularities of $\mathcal{X}_0$, where $\pi$ is also nondegenerate. Then standard Schauder estimates and the linear estimates after linearizing the Monge-Amp\`ere equation (\ref{famke})  can be established locally, which gives uniform higher order regularity for $\varphi_t$. 

\end{proof}

For any nonsingular point $p_0$ in $\mathcal{X}_0$, there exists an open neighborhood $U$ of $p_0$ in $\mathcal{X}$ such that $\mathcal{X}_t\cap U$  are all biholomorphic to each other and $\chi_t|_{\mathcal{X}_t \cap U}$ are all equivalent for all $t\in B$.  We then can pick any sequence $t_j \in B^*$ converging to $0$. By the uniform estimates for $\varphi_t$, $\varphi_{t_j}$ converges smoothly to a smooth function $\varphi_0, $ on $\mathcal{R}_{\mathcal{X}_0}$. Furthermore, $\varphi_0$ satisfies the following conditions. 

\begin{enumerate}

\item There exists $C>0$ such that $$ \sup_{\mathcal{R}_{\mathcal{X}_0}} \varphi_0 \leq C. $$

\item For any $p \in \mathcal{R}_{\mathcal{X}_0}$, there exists an effective divisor $G_p$ numerically equivalent to $K_{\mathcal{X}_0}$ such that $G_p$ does not vanish at $p$ and $G_p$ contains $\textnormal{LCS}(\mathcal{X}_0)$. For any $\epsilon>0$, there exists $C_\epsilon>0$
$$\varphi_0 - \epsilon \log |\sigma_{G_p}|^2_{h_{\Omega_0}} \geq - C_\epsilon, $$
where $h_{\Omega_0} = \left( \Omega|_{\mathcal{X}_0}\right)^{-1} $ is a hermtian metric on $K_{\mathcal{X}_0}$,  where $\chi_0= \chi|_{\mathcal{X}_0}$. This estimate follows from Lemma \ref{54} as $\varphi_{t'}$ is milder by any log poles along the exceptional divisor and the normal crossings among components of $\widetilde{\mathcal{X}_0}$.

\medskip

\item $\varphi_0$ solves the following equation on $\mathcal{R}_{\mathcal{X}_0}$
$$(\chi_0 + \ddbar \varphi_0)^n = e^{\varphi_0}\Omega_0, $$

\end{enumerate} 

By the uniqueness in Lemma \ref{uniq}, $\varphi_0$ must coincide with the unique solution constructed in Lemma \ref{existsol} and Lemma \ref{uniq}. Hence we have established the following lemma. 

\begin{lemma} Let $\omega_t = \chi_t + \ddbar \varphi_t$ be the K\"ahler-Einstein metric on $\mathcal{X}_t$, $t\in B^*$ with $$(\chi_t + \ddbar \varphi_t)^n = e^{\varphi_t} \Omega_t.$$
Then $\varphi_t$ converges to a unique a unique $\varphi_0 \in PSH(\mathcal{X}_0, \chi_0)\cap C^\infty(\mathcal{R}_{\mathcal{X}_0}) \cap L^\infty_{loc}(\mathcal{X}_0\setminus \textnormal{LCS}(\mathcal{X}_0))$, where $\textnormal{LCS}(\mathcal{X}_0)$ is the non-log terminal locus of $\mathcal{X}_0$. 
Furthermore, the following holds.

\begin{enumerate}

\item $\omega_0 = \chi_0 + \ddbar \varphi_0$ is  K\"ahler-Einstein current on $\mathcal{X}_0$ with 
$$(\chi_0+ \ddbar \varphi_0)^n = e^{\varphi_0} \Omega_0.$$

\item $\varphi_0$ tends to $-\infty$ near $\textnormal{LCS}(\mathcal{X}_0)$.

\medskip

\item  $\int_{\mathcal{R}_{\mathcal{X}_0}} \omega_0^n = [K_{\mathcal{X}_t} ]^n$, for all $t\in B$. 

\end{enumerate} 
In particular, $\varphi_0$ coincides with the unique solution in Lemma \ref{uniq}. 

\end{lemma}

The following lemma establishes the uniform non-collapsing condition for the K\"ahler-Einstein manifolds $(\mathcal{X}_t, g_t)$, for all $t\in B^*$. 

\begin{lemma} \label{noncol} For any nonsingular point $p_0\in \mathcal{X}_0$, we can pick a smooth section $p(t): B \rightarrow \mathcal{X}$ such that $p(0) =p_0$. Then  there exists $c >0$ such that for all $t\in B^*$, 
\begin{equation}
Vol_{g_t} (B_{g_t}(p(t), 1)) \geq c,
\end{equation}
where $B_{g_t}(p(t), 1)$ is the unit geodesic ball centered at $p(t)$ in $(\mathcal{X}_t, g_t)$.

\end{lemma}

\begin{proof} Since $p_0$ is a nonsingular point of $\mathcal{X}_0$, there exists $r>0$ such that the geodesic ball $B_{\chi_0}(p_0, r) $ in $(\mathcal{X}_0, \chi_0)$ completely lies in the nonsingular part of  $\mathcal{X}_0$. Then there exists $c>0$ such that for all $t\in B$, 
$$Vol_{\chi_t} (B_{\chi_t}(p(t), r))>c. $$
By Lemma \ref{sch1},   there exists $C>0$ such that for all $t\in B^*$, 
$$C^{-1} \chi_t \leq \omega_t \leq C \chi_t$$
 on $B_{\chi_t}(p(t), r)$ in $(\mathcal{X}_t, \chi_t)$. This implies that 
$$B_{\chi_t} (p(t), r) \subset B_{g_t}(p(t), C^{1/2} r) $$
and so 
$$Vol_{g_t}(B_{g_t}(p(t), C^{1/2} r) \geq C^{-n} Vol_{\chi_t} (B_{\chi_t}(p(t), r)) \geq  C^{-n} c. $$
The lemma then immediately follows by the volume comparison theorem. 

\end{proof}

The proof of Lemma \ref{noncol} also shows that for any  given nonsingular point $p_0$ in $\mathcal{X}_0$, the non-collapsing condition holds uniformly for all points near  $p_0$.


\section{Gromov-Hausdorff convergence and partial $C^0$-estimates}

Let $\pi: \mathcal{X}\rightarrow B$ be a stable degeneration of smooth canonical models as considered in Section 4. Suppose
$$\mathcal{X}_0= \bigcup_{\alpha=1}^\mathcal{A} X_\alpha,$$
where each $X_\alpha$ is an irreducible component of $\mathcal{X}_0$. We pick any $\mathcal{A}$-tuple of nonsingular points 
$$(p_0^{1}, p_0^{2}, ..., p_0^{\mathcal{A}}), ~~p_0^{\alpha} \in X_\alpha \cap \mathcal{R}_{\mathcal{X}_0}, ~\alpha=1, ..., \mathcal{A}.$$

Let $(p_{t_j}^{1}, p_{t_j}^{2}, ..., p_{t_j}^{\mathcal{A}})$  a  sequence of $\mathcal{A}$-tuples of points $\in \mathcal{X}_{t_j}$ with $t_j \rightarrow 0$ such that  
$$(p_{t_j}^{1}, p_{t_j}^{2}, ..., p_{t_j}^{\mathcal{A}})  \rightarrow (p_0^{1}, p_0^{2}, ..., p_0^{\mathcal{A}}) $$
 with respect to the fixed reference metric $\chi$ on $\mathcal{X}$.  Let $g_{t_j}$ be the corresponding K\"ahler-Einstein metric on $\mathcal{X}_{t_j}$. We would like to study the Riemannian geometric convergence of $(\mathcal{X}_{t_j}, g_{t_j})$ as $t_j \rightarrow 0$.

\begin{lemma}  \label{GHcon}
Let $g_t$ be the unique K\"ahler-Einstein metric on $\mathcal{X}_t$ for $t\in B^*$.  
%
%
After passing to a subsequence, $(\mathcal{X}_{t_j},  g_{t_j}, (p_{t_j}^1, ..., p_{t_j}^\mathcal{A}))$ converges in pointed Gromov-Hausdorff topology to a metric length space 
$$(\mathbf{Y}, d_{\mathbf{Y}})= \coprod_{\beta=1}^\mathcal{B} (Y_\beta, d_\beta)$$
as a disjoint union of  metric length spaces $(Y_\beta, d_\beta),$ 
satisfying
\begin{enumerate}

\item $\mathbf{Y} = \mathcal{R}_{\mathbf{Y}} \cup \mathcal{S}_{\mathbf{Y}}$, where $\mathcal{R}_{\mathbf{Y}}$ and $\mathcal{S}_{\mathbf{Y}}$ are the regular and singular part of $\mathbf{Y}$.
 $\mathcal{R}_{\mathbf{Y}}$ is an open K\"ahler manifold  and $\mathcal{S}_{\mathbf{Y}}$ is closed of Hausdorff dimension no greater than $2n-4$. 

\medskip

\item $g_{t_j}$ converge smoothly to a K\"ahler-Einstein metric $g_{KE}$ on $\mathcal{R}_{\mathbf{Y}}$. In particular, $g_{KE}$ coincides with the unique K\"ahler-Einstein current constructed in Theorem \ref{main1} on $\mathcal{X}_0$. 

\medskip

\item $\mathcal{R}_{\mathcal{X}_0}$ is an open dense set in $(\mathbf{Y}, d_{\mathbf{Y}})$ and 
$$\mathcal{R}_{\mathcal{X}_0} \subset \mathcal{R}_{\mathbf{Y}}. $$

\medskip

\item  $\mathcal{B} \leq \mathcal{A}$ and 
$$ \textnormal{Vol}(\mathbf{Y}, d_{\mathbf{Y}}) = \sum_{\beta=1}^\mathcal{B} \textnormal{Vol}(Y_\beta, d_\beta) = \textnormal{Vol}(\mathcal{X}_t, g_t) $$ 
for all $t\in B^*$. 

\end{enumerate}

\end{lemma}

\begin{proof}  By Lemma \ref{noncol}, there exists $c>0$ and  $>0$ such that for all $\alpha$ and $j$, 
$$Vol_{g_{t_j}} (B_{g_{t_j}}(p^\alpha_{t_j}, 1) \geq c.$$
Then Cheeger-Colding-Tian theory \cite{CCT} immediately implies that the pointed Gromov-Hausdorff convergence and (1), (2) hold. 

Since $g_t$ converges smoothly on the nonsingular part $\mathcal{R}_{\mathcal{X}_0}$, $\mathcal{R}_{\mathcal{X}_0}$ must be contained in the regular part $\mathcal{R}_{\mathbf{Y}}$ of $(\mathbf{Y}, d_{\mathbf{Y}})$.

The volume  $Vol_{g_t}(\mathcal{X}_t ) = [K_{\mathcal{X}_t}]^n$ is a fixed constant for all $t\in B$, so by the volume convergence, we have 
$$ \textnormal{Vol}(\mathbf{Y}, d_{\mathbf{Y}}) \leq Vol_{g_t}(\mathcal{X}_t) = [K_{\mathcal{X}_t}]^n  $$
for all $t\in B^* $. On the other hand, $g_{KE}$ extends to the unique K\"ahler-Einstein current on $\mathcal{X}_0$ and so by Lemma \ref{existsol} and Corollary \ref{uniq}, 
$$\textnormal{Vol}(\mathbf{Y}, d_{\mathbf{Y}}) \geq \int_{\mathcal{R}_{\mathcal{X}_0}} dV_{g_{KE}} = [K_{\mathcal{X}_0}]^n. $$
This implies that $\mathcal{R}_{\mathcal{X}_0}$ must be dense in $(\mathbf{Y}, d_{\mathbf{Y}}).$

Since $\mathcal{R}_{\mathcal{X}_0}$ has $\mathcal{A}$ disjoint components with total volume equal to the volume of $(\mathbf{Y}, d_{\mathbf{Y}})$, there can be at most $\mathcal{A}$ disjoint components in $(\mathbf{Y}, d_{\mathbf{Y}})$. 

Therefore we have proved (4) and (5).

\end{proof}

%




 The following proposition can be proved by similar arguments in \cite{T1, DS1} as local $L^2$-estimates from Tian's proposal for the partial $C^0$-estimates. We let $h_t = ((\omega_t)^n)^{-1}= (e^{\varphi_t}\Omega_t )^{-1}$ be the hermitian metric on $\mathcal{X}_t$ for $t\in B^*$, where $\omega_t$ is K\"ahler-Einstein form associated to the Kahelr-Einstein metric $g_t$ on $\mathcal{X}_t$, $\Omega_t$ and $\varphi_t$ are defined in Section 4. 

\begin{lemma} \label{l41}   Let $p_0$ be a nonsingular point in $\mathcal{X}_0$ and $p: B \rightarrow \mathcal{X}$ be a smooth section with $p(0)=p_0$.  For any $R>0$, there exists  $K_R >0$ such that if  $\sigma \in H^0(\mathcal{X}_t, mK_{\mathcal{X}_t})$ for $m\geq 1$ with $t\in B^*$, then 
\begin{equation}\label{l51}
\|\sigma \|_{L^{\infty, \sharp}(B_{g_t}(p(t), R))} \leq K_R \| \sigma \|_{L^{2, \sharp}(B_{g_t}(p(t), 2R) )}
\end{equation}
\begin{equation}\label{l52}
 \|\nabla \sigma \|_{L^{\infty, \sharp}(B_{g(t)}(p(t), R)) } \leq K_R \|\sigma \|_{L^{2, \sharp}(B_{g_t}(p(t), 2R))}, 
\end{equation}
where $B_{g_t}(p(t), R)$ is the geodesic ball centered at $p(t)$ with radius $R$ in $(\mathcal{X}_t, g_t)$, the $L^2$-norms $||\sigma||_{L^{\infty,\sharp}}$ and  $||\nabla \sigma||_{L^{\infty,\sharp }}$ are defined with respect to the rescaled hermitian metric $(h_t)^m$ and the rescaled K\"ahler metric $mg_t$. 

\end{lemma}

\begin{proof} For fix $R>0$, the Sobolev constant on $B_{g_t}(p(t), R)$ is uniformly bounded because of the Einstein condition and the uniform noncollapsing condition for unit balls centered at $p(t)$. The proof follows by well-known argument of Moser's iteration on balls of relative scales using cut-off functions (c.f \cite{S2}).

\end{proof}

The following $L^2$-estimate is standard for K\"ahler-Einstein manifolds.

\begin{lemma} For any integer $m\geq 2$, any $t\in B^*$ and any smooth $(m K_{\mathcal{X}_t})$-valued $(0,1)$-form $\tau$ satisfying 
$\dbar \tau =0$,  
there exists an $(mK_{\mathcal{X}_t})$-valued section $u$ such that $\dbar u = \tau$ and $$ \int_{\mathcal{X}_t} |u|^2_{(h_t)^m} ~dV_{g_t} \leq \frac{1}{2\pi} \int_{\mathcal{X}_t} |\tau|^2_{(h_t)^m}~ dV_{g_t}.$$

\end{lemma}

The following lemma gives a construction for global pluricanonical section on the limiting metric space $\mathbf{Y}$.

\begin{lemma} \label{ext1} Suppose $t_j \in B^* \rightarrow 0$ and $\sigma_{t_j} \in H^0(\mathcal{X}_t, mK_{\mathcal{X}_t})$ be a sequence of sections satisfying 
$$\int_{\mathcal{X}_{t_j}} |\sigma_{t_j}|^2_{(h_{t_j})^m} dV_{mg_{t_j}} = 1. $$
Then after passing to a subsequence, $\sigma_{t_j}$ converges to a holomorphic section $\sigma$ of $mK_{\mathbf{Y}}$. Furthermore, the $\sigma|_{\mathcal{R}_{\mathcal{X}_0}}$ extends to a unique $\sigma' \in H^0(\mathcal{X}_0, mK_{\mathcal{X}_0})$ and $\sigma$ vanishes along $\textnormal{LCS}(\mathcal{X}_0)$.

\end{lemma}

\begin{proof} We choose any nonsingular point $p_0$ in $\mathcal{X}_0$ and a sequence $p_{t_j}\in \mathcal{X}_{t_j}$ such that $p_{t_j}$ converges to $p_0$ in $(\mathcal{X}, \chi)$. Then $p_{t_j}$ in $(\mathcal{X}_{t_j}, g_{t_j})$ also converges in Gromov-Hausdorff distance to $p_0$ in $(\mathbf{Y}, d_{\mathbf{Y}})$ due to the smooth convergence of $g_{t_j}$ to $g_{KE}$ on $\mathcal{R}_{\mathcal{X}_0}$. Since the $L^2$-norm of $\sigma_{t_j}$ with respect to $(h_{t_j})^m$ and $mg_{t_j}$ is uniformly bounded, for any $R>0$, there exists $C_R>0$ such that for all $j$, we have
\begin{equation} 
\sup_{B_{g_{t_j}}(p_{t_j}, R) } |\sigma_{t_j} |\leq C_R, ~  \sup_{B_{g_{t_j}}(p_{t_j}, R) } |\nabla \sigma_{t_j} |_{g_{t_j}} \leq C_R . 
\end{equation}
After passing to sequence, $\sigma_{t_j}$ converges to a section $\sigma$ of $mK_{\mathbf{Y}}$ and it is a holomorphic on $\mathcal{R}_{\mathbf{Y}} \supset \mathcal{R}_{\mathcal{X}_0}$. It uniquely extends to the singular set of $\mathbf{Y}$ because of the gradient estimate and the geodesic convexity of $\mathcal{R}_{\mathbf{Y}}$.

Since $\sigma$ is a holomorphic section of $mK_{\mathcal{R}_{\mathcal{X}_0}}$, it can be uniquely extended to a pluricanonical section on the normal part of $\mathcal{X}_0$. Without loss of generality, we can assume that the adapted volume measure is given by $\Omega_0= (\sum_{j=0}^N \eta_j \wedge \overline{ \eta_j} )^{1/m}$, where $\{\eta_j\}$ gives a global projective embedding of $\mathcal{X}_0$. Let $\omega_0 = \chi_0 + \ddbar \varphi_0$ be unique the K\"ahler-Einstein current on $\mathcal{X}_0$ constructed in Lemma \ref{existsol} and Corollary \ref{uniq}, satisfying 
$$(\chi_0+ \ddbar \varphi_0)^n = e^{\varphi_0} \Omega_0. $$
By the $L^2$-bound of $\sigma$ and the upper bound of $\varphi_0$ from Corollary \ref{5upb},  there exists $C>0$ such that 
\begin{equation}\label{integr}
\int_{\mathcal{R}_{\mathcal{X}_0}} \frac{|\sigma|^2}{(\Omega_0)^m}  \Omega_0\leq C \int_{\mathcal{R}_{\mathcal{X}_0}} \frac{|\sigma|^2}{(\Omega_0)^m}  e^{-m\varphi_0} \Omega_0 < \infty. 
\end{equation}
We can pullback the above formula on $\widetilde{\mathcal{X}_0}$, a log resolution of $\mathcal{X}_0$. The pullback of $\Omega_0$ on $\widetilde{\mathcal{X}_0}$ has poles of order $1$ along the exceptional divisor over the non-log terminal  locus $\textnormal{LCS}(\mathcal{X}_0) $ of $\mathcal{X}_0$. Therefore  $\sigma$ must vanish  along $\textnormal{LCS}(\mathcal{X}_0) $ so that the integral (\ref{integr}) is finite. More precisely, for any point $x \in \textnormal{LCS}(\mathcal{X}_0)$, there exists an open neighborhood $U_x$ of $x$, such that $\frac{\sigma}{\eta_i}$ is bounded on $U_x$ for some $i$, and 
$$\left. \frac{\sigma}{\eta_i} \right|_{U_x \cap \textnormal{LCS}(\mathcal{X}_0) } = 0. $$
As a consequence, $\sigma$ can be uniquely extended globally to $\sigma'$ on $\mathcal{X}_0$.

\end{proof}

The following is the local version of the partial $C^0$-estimate. 
\begin{lemma} \label{par0}  Let $p_0\in \mathcal{R}_{\mathcal{X}_0}$ and $p: B \rightarrow \mathcal{X}$ be a smooth section with $p(0)=p_0$.  For any $R>0$, there exist $m \in \mathbb{Z}^+$ and $c>0$ such that for any $t\in B^*$ and $q\in B_{g_t}(p(t), R)$, there exists $\sigma_t \in H^0(\mathcal{X}_0,  mK_{\mathcal{X}_t})$ satisfying
\begin{equation}
 |\sigma_t|^2_{ (h_t)^m} (q)\geq c,  ~~~ \int_{\mathcal{X}_t} |\sigma_t|^2_{(h_t)^m} dV_{mg_t} = 1. 
\end{equation}

\end{lemma}

\begin{proof} The proof of the global partial $C^0$-estimate in \cite{DS1} can be directly applied here in the local case with the estimates in Lemma \ref{l41} and Lemma \ref{ext1} because the singular set of all iterated tangent cones of the Gromov-Hausdorff limit $(\mathbf{Y}, d_{\mathbf{Y}} )$ is closed and has Hausdorff dimension less than $2n-2$.

\end{proof}

Let $h_{\mathbf{Y}}$ be the hermitian metric on $K_\mathbf{Y}$ as the extension of $((\omega_{\mathbf{Y}})^n)^{-1}$ from $\mathcal{R}_{\mathbf{Y}}$, where $\omega_{\mathbf{Y}}$ is the K\"ahler-Einstein form on $\mathcal{R}_{\mathbf{Y}}$. We now pass the partial $C^0$-estimate in Lemma \ref{par0} to the limiting space $\mathbf{Y}$.

\begin{corollary} \label{par1} For any $p_0 \in \mathcal{R}_{\mathbf{Y}}$ and any $R>0$, there exist $m \in \mathbb{Z}^+$ and $c, C>0$ such that for any $q\in B_{d_{\mathbf{Y}}}(q, R)$, there exists $\sigma \in H^0(\mathbf{Y},  mK_{\mathbf{Y}})$ satisfying
\begin{equation}
 |\sigma|^2_{ (h_{\mathbf{Y}})^m} (q)\geq c,  ~~~ \int_{\mathbf{Y}} |\sigma|^2_{(h_{\mathbf{Y}})^m} dV_{d_{\mathbf{Y}}} = 1. 
\end{equation}

\end{corollary}
We remark that such $\sigma$ has uniformly bounded gradient estimate and it can also be extended to a global pluricanonical section on $\mathcal{X}_0$. There are many generalizations and variations of Lemma \ref{par0} and Corollary \ref{par1}. For example, if $p$ is a regular point in $\mathbf{Y}$, then there exist global pluricanonical sections $\sigma_0$, ..., $\sigma_n$ such that near $p$, $\sigma_0$ is nonzero and 
$$\frac{\sigma_1}{\sigma_0}, ..., \frac{\sigma_n}{\sigma_0}$$
can be used as holomorphic local coordinates near $p$.  The proof of Theorem \ref{integr} and Corollary \ref{par1} is also used to construct peak sections to separate distinct points on $\mathbf{Y}$.


\section{Distance estimates}

In this section, our goal is to estimate the distance from a singular point of $\mathcal{X}_0$ to a given nonsingular point of $\mathcal{X}_0$. We will establish a principle for geometric complex Monge-Amp\`ere equations of our interest that boundedness of local potentials is equivalent to boundedness of distance.

First,  we want to show that the regular part of the Gromov-Hausdorff limit coincides with the nonsingular part of $\mathcal{X}_0$.

\begin{lemma} \label{regularid}
Let $\mathcal{R}_{\mathbf{Y}}$ be the regular part of the metric space $(\mathbf{Y}, d_{\mathbf{Y}})$ and $\mathcal{R}_{\mathcal{X}_0}$ be the nonsingular part of the projective variety of $\mathcal{X}_0$. Then
$$\mathcal{R}_{\mathbf{Y}} = \mathcal{R}_{\mathcal{X}_0}$$
and they are biholomorphic to each other. 

\end{lemma}

\begin{proof} Obviously, $$\mathcal{R}_{\mathbf{Y}} \supset \mathcal{R}_{\mathcal{X}_0}$$ from Lemma \ref{GHcon} and it suffices to show the other direction.

We prove by contradiction. Suppose $q\in \mathcal{R}_{\mathbf{Y}} \setminus \mathcal{R}_{\mathcal{X}_0}$ in $(\mathbf{Y}, d_{\mathbf{Y}})$. Since $\mathcal{R}_{\mathcal{X}_0}$ is dense in $\mathbf{Y}$, there exist a sequence of points $q_j \in \mathcal{R}_{\mathcal{X}_0}$ such that $q_j \rightarrow q$ with respect to $d_{\mathbf{Y}}$ in $\mathbf{Y}$.  We can assume after passing to a subsequence that $q_j$ converges to $q' \in \mathcal{X}_0$ with respect to $\chi_0$ in $\mathcal{X}_0$. 

Since $q$ is a regular point in $(\mathbf{Y}, d_{\mathbf{Y}})$, there exists $r>0$ such that $B_{d_{\mathbf{Y}}} (q, r)$ is a smooth open domain in the Euclidean space $\mathbb{C}^n$, and $d_{\mathbf{Y}}$ induces a smooth K\"ahler metric on $B_{d_{\mathbf{Y}}}(q, r)$.  Furthermore, there exist $t_i \rightarrow 0$  and $x_i \in \mathcal{X}_{t_i}$ such that, $B_{g_{t_i}}(x_i, r) \subset \mathcal{X}_{t_i}$ converges smoothly to $B_{d_{\mathbf{Y}}}(q, r)$.  By the partial $C^0$-estimates for regular points, there exist global sections $\sigma_0, \sigma_1, ..., \sigma_n$ of $H^0(\mathbf{Y}, m K_{\mathbf{Y}})$ for some sufficiently large $m\in \mathbb{Z}^+$ such that after making $r$ sufficiently small, $\sigma_0$ does not vanish on $B_{d_{\mathbf{Y}}}(q, r)$ and 
$$\frac{\sigma_1}{\sigma_0}, ..., \frac{\sigma_n}{\sigma_0}$$
are holomorphic local coordinates on $B_{d_{\mathbf{Y}}}(q, r)$. In particular, the $L^2$-norm of each $\sigma_k$ is bounded with respect to $h_\mathbf{Y}$ and $d_\mathbf{Y}$. Therefore we have a biholomrphism


%
%
%
%
%
%
\begin{equation}\label{imv}
F: B_{d_\mathbf{Y}}(q, r) \rightarrow \mathcal{V} \in \mathbb{C}^n 
\end{equation}
by letting $F=\left( \frac{\sigma_1}{\sigma_0}, ..., \frac{\sigma_n}{\sigma_0} \right) $.

Also each $\sigma_k$ can be uniquely extended to a pluricanonical section on  $\mathcal{X}_0$, 
for $k=0, 1, ...n$.  We write it as $\sigma'_k$. As before, we can assume that the adapted volume measure $\Omega_0= \left(\sum_{l=0}^N \eta_l \wedge \overline{\eta_l } \right)^{1/m}$, where $\{\eta_l\}_{l=0}^N $ gives a global projective embedding of $\mathcal{X}_0$. 
We will discuss in the following two cases. 

\begin{enumerate}


\item Suppose $q' \in \textnormal{LCS}(\mathcal{X}_0)$. The K\"ahler-Einstein volume form on $B_{d_\mathbf{Y}}(q, r)$ is given by
$$|\sigma_0|^{2/m} e^{\varphi_{\mathcal{V}}}, $$
where $\varphi_{\mathcal{V}}$ is a smooth bounded plurisubharmonic function on $B_{d_\mathbf{Y}}(q, r)$.  Therefore on $\mathcal{R}_{\mathcal{X}_0}\cap B_{d_\mathbf{Y}}(q, r)$, we have
$$|\sigma_0|^{2/m} = \left( \sum_{l=0}^N \eta_l \wedge \overline{ \eta_l} \right)^{1/m} e^{\varphi_0- \varphi_{\mathcal{V}}},  $$
where $\varphi_0$ satisfies the Monge-Amp\`ere equation 
$$(\chi_0+\ddbar\varphi_0)^n = e^{\varphi_0} \Omega_0. $$
Suppose $\eta_0$ generates  $m K_{\mathcal{X}_0}$ at $q'$. Then
$$f=\frac{\sigma'_0}{\eta_0}$$
is a meromorphic function on $\mathcal{X}_0$, where $\sigma'_k$ is the unique extension of $\sigma_k$ from $\mathcal{R}_{\mathcal{X}_0}$ to $\mathcal{X}_0$ for $k=0, 1, ..., n$.  $f$ is a holomorphic function in a neighborhood $U_{q'}$ in $\mathcal{X}_0$ and $f(q')=0$ since $\sigma'_0$ vanishes at $q'$. We can replace $U_{q'}$ by a smooth neighborhood $\widetilde{U_{q'}}$ after a log resolution of $\mathcal{X}_0$. Then the pullback of $f$ on $U_{q'}$ vanishes along a divisor $E$ (containing the exceptional divisor over $\textnormal{LCS}(\mathcal{X}_0) )$ and so it must vanish to order of at least $1$. On the other hand, after pulling back $\varphi_0$ to $\widetilde{U_{q'}}$,  for any $\epsilon>0$ there exists $C_\epsilon>0$ such that at each pre-image of $q_j$ in $\widetilde{U_{q'}}$, 
$$|f|^2 \geq e^{\varphi_0 - \varphi_{\mathcal{V}}} \geq C_\epsilon |f|^{2\epsilon} $$  because $\varphi_0$ is bounded below any log poles along the exceptional divisor over $\textnormal{LCS}(\mathcal{X}_0)$ and $\varphi_{\mathcal{V}}$ is uniformly bounded at $q_j$. Then we have
$$\liminf_{j\rightarrow \infty} |f|(q_j) >0. $$ This contradicts the fact that $f(q')=0$.

\medskip

\item  $q' \notin \textnormal{LCS}(\mathcal{X}_0)$.  Without loss of generality, we can assume that  there exists $C>0$ such that for any $x\in B_{d_\mathbf{Y}}(q, r) \cap\mathcal{R}_{\mathcal{X}_0}$, 
$$d_{\chi} (x, \textnormal{LCS}(\mathcal{X}_0)) > C. $$
Otherwise, there exist a sequence of $x_j \in  B_{d_\mathbf{Y}}(q, r)  \cap \mathcal{R}_{\mathcal{X}_0}$ such that $x_j$ converges to some $x' \in \textnormal{LCS}(\mathcal{X}_0)$ in $(\mathcal{X}_0, \chi_0)$. This can be reduced to the previous case. 

Immediately, we can show the extension $\sigma'_0$ on $\mathcal{X}_0$ from $\sigma_0$ does not vanish near $q'$ because $|\sigma'_0|^{2/m} e^{\varphi_{\mathcal{V}} }= \Omega_0 e^{\varphi_0}$ and since $\varphi_{\mathcal{V}}(q_j)$ and $\varphi_0(q_j)$ are both uniformly bounded for all $j$. Since $\frac{\sigma_k (q_j)}{\sigma_0(q_j)}=\frac{\sigma_k' (q_j)}{\sigma_0'(q_j)}$ converges to $\frac{\sigma_k(q)}{\sigma_0(q)}$, $\frac{\sigma_k'(q')}{\sigma'_0(q')} = \frac{\sigma_k(q)}{\sigma_0(q)}$. Therefore we also obtain a holomorphic map 
%
%
%
$$F' = \left(\frac{\sigma_1'}{\sigma'_0}, ..., \frac{\sigma'_n}{\sigma'_0} \right): \mathcal{X}_0 \setminus \{ \sigma'_0=0\}\rightarrow \mathbb{C}^n, ~ F'(q') = F(q).$$
Let $U = (F')^{-1} (\mathcal{V})$, where $\mathcal{V}$ is given by (\ref{imv}). Then $U$ is an open set in $\mathcal{X}_0$ and $q' \in U$. 

Since the nonsingular part of $\mathcal{X}_0$ is open and dense in $\mathbf{Y}$, there exists a subvariety $D$ of $\mathcal{V}$ such that $$F'|_{(F')^{-1}(\mathcal{V}\setminus D)}: (F')^{-1}(\mathcal{V}\setminus D) \rightarrow \mathcal{V}\setminus D$$  is biholomorphic. 
%

Recall $\{\eta_0, ..., \eta_N \}$ gives a projective embedding of $\mathcal{X}_0$. There exists a   sufficiently small open neighborhood $U_{q'}$ of $q'$ in $\mathcal{X}_0$ such that
$$\gamma= \left(\frac{\eta_1}{\sigma'_0}, ..., \frac{\eta_N}{\sigma'_0} \right): U_{q'} \rightarrow \mathbb{C}^N $$ is  a local affine embedding, where we assume $\eta_0$ does not vanish on $U_{q'}$ (we can always shrink $U_{q'}$). 
We can also assume $F'(U_{q'}) \subset \mathcal{V}$ by continuity of $F'$ since $\sigma'_0(q')\neq 0$ and $F'(q')\in \mathcal{V}$. 
We now consider the map $F''$ defined by
$$ F'' =  \left(\frac{\sigma_1'}{\sigma'_0}, ..., \frac{\sigma'_n}{\sigma'_0}, \frac{\eta_1}{\sigma'_0}, ..., \frac{\eta_N}{\sigma'_0} \right): U_{q'} \rightarrow \mathbb{C}^n\times \mathbb{C}^N . $$
Clearly, $F''$ is also an affine embedding of $U_{q'}$. 
We identify $\frac{\eta_i}{\sigma'_0}$ on $(F')^{-1}(\mathcal{V}\setminus D) $ and $\frac{\eta_i}{\sigma'_0} \circ (F'^{-1})$ on   $\mathcal{V}\setminus D$. Each $\frac{\eta_i}{\sigma'_0}$ is a holomorphic function on $\mathcal{V}\setminus D$ and 
$$\sup_{\mathcal{V}\setminus D} \frac{ \left( \sum_{i=0}^N \eta_i \wedge \overline{\eta_i}\right)^{1/m}}{ |\sigma'_0|^{2/m} } = \sup_{\mathcal{V}\setminus D} e^{\varphi_0 - \varphi_{\mathcal{V}}} < \infty $$
Therefore each $\frac{\eta_i}{\sigma'_0}$ is uniformly bounded on $\mathcal{V}$ and it extends to a holomorphic function in $\mathcal{V}$. In particular, each $\frac{\eta_i}{\sigma'_0}$ is a holomorphic function  in $\left(\frac{\sigma_1'}{\sigma'_0}, ..., \frac{\sigma'_n}{\sigma'_0} \right)$ on $\mathcal{V}$ for $i=0, 1, ..., N$. So 
$F''$ is an isomorphism between $U_{q'}$ and $F'(U_{q'})$. But $F'(U_{q'})$ lies in the graph of $\mathcal{V}$ of $\left(\frac{\eta_1}{\sigma'_0}, ..., \frac{\eta_N}{\sigma'_0} \right)$ in $\mathbb{C}^n \times \mathbb{C}^N$ and so $F'(U_{q'})$ must be nonsingular. Contradiction.

\end{enumerate}

\end{proof}

\begin{corollary}

Let $\mathcal{A}$ be the number of the component of $\mathcal{X}_0$ and $\mathcal{B}$ the number of the components of $\mathbf{Y}$. 
 Then $$ \mathcal{A} = \mathcal{B}.$$
\end{corollary}

\begin{proof} Let $p$ and $p'$ be two nonsingular points in distinct components $X$ and $X'$ of $\mathcal{X}_0$. Then there exists a smooth minimal geodesic $\gamma$ joining $p$ and $p'$ in $(\mathbf{Y}, d_{\mathbf{Y}})$ by the geodesic convexity result in \cite{CN}, and $\gamma$ lies entirely in $\mathcal{R}_{\mathbf{Y}}=\mathcal{R}_{\mathcal{X}_0}$. The corollary is proved because $\mathcal{R}_{\mathcal{X}_0}\cap (X\cup X')$  is not connected.

\end{proof}






We now prove another version of Schwarz lemma with suitable barriers. 

\begin{lemma} \label{sch4} Let $\omega_0=  \chi_0+ \ddbar \varphi_0$ be the unique K\"ahler-Einstein current on $\mathcal{X}_0$. For any compact set $K \subset\subset \mathcal{X}_0\setminus \textnormal{LCS}(\mathcal{X}_0) $, there exists $c=c(K) >0$ such that 

$$\omega_0 \geq c \chi_0 $$ on $K\cap \mathcal{R}_{\mathcal{X}_0}$.

\end{lemma}

\begin{proof}

 Since $\mathcal{X}_0$ is a semi-log canonical model, by Lemma \ref{effdiv}, for any $p \in \mathcal{X}_0 \setminus \textnormal{LCS}(\mathcal{X}_0)$, there exists an effective $\mathbb{Q}$-divisor $G_p$ numerically equivalent to $K_{\mathcal{X}_0}$ such that $p$ does not lie in $G_p$ and the support of  $\textnormal{LCS}(\mathcal{X}_0)$ is contained in the support of $G_p$. Let $\sigma_{G_p}$ be the defining divisor of $G_p$ and $h_{\Omega_0}$ be the hermitian metric on $G_p$ so that $Ric(h_{\Omega_0})=\chi_0 = \chi|_{\mathcal{X}_0}$. 

By similar argument in Lemma \ref{effdiv}, we can also find an effective $\mathbb{Q}$-divisor $\mathcal{F}$ numerically equivalent to $K_{\mathcal{X}_0}$ such that   the support of  the singular set $\mathcal{S}_{\mathcal{X}_0}$ of $\mathcal{X}_0$ is contained in the support of $\mathcal{F}$. Let $\sigma_{\mathcal{F}}$ be the defining divisor of $\mathcal{F}$.

Let $\omega_0 = \chi_0 + \ddbar \varphi_0$ be the unique K\"ahler-Einstein current on $\mathcal{X}_0$. By Lemma \ref{existsol}, there exists $C_1>0$ and for any $\epsilon>0$, there exists $C_2= C_2(\epsilon)>0$ such that 
$$ -C_2 + \epsilon \log |\sigma_{G_p}|^2_{h_{\Omega_0}} \leq \varphi_0 \leq C_1. $$

We now define
$$H_{\epsilon, \epsilon'}  = \log tr_{\omega_0}(\chi_0) - 3A\varphi_0+ \epsilon \log |\sigma_{G_p}|^2_{h_{\Omega_0}} + \epsilon' \log |\sigma_{\mathcal{F}}|^2_{h_{\Omega_0}}.$$ 
By Lemma \ref{sch1}, for any sufficiently small $\epsilon'>0$, we have
$$\sup_{\mathcal{R}_{\mathcal{X}_0}} |\sigma_{F}|^{2\epsilon'}_{h_{\Omega_0}} tr_{\omega_0}(\chi_0) <\infty$$
and  $|\sigma_{F}|^{2\epsilon'}_{h_{\Omega_0}} tr_{\omega_0}(\chi_0)$ tends $-\infty$ near the support $\mathcal{F}$.

Straightforward calculations show that  on the nonsingular part of $\mathcal{X}_0$, for a fixed sufficiently large $A>0$, there exists $C_3>0$ such that  for all sufficiently small $\epsilon, \epsilon'>0$, 
$$\Delta_{\omega_0} H_{\epsilon, \epsilon'} \geq   A tr_{\omega_0}(\chi_0) - C_3. $$
We can apply the maximum principle for $H_{\epsilon, \epsilon'}$ at its maximal point, which must lie in the nonsingular part of $\mathcal{X}_0$. By the estimate for $\varphi_0$, there exists $C_4=C_4(\epsilon)>0$ such that 
$$H_{\epsilon, \epsilon'}  \leq C_4 $$
and so on $\mathcal{R}_{\mathcal{X}_0}$, there exists $C_5=C_5(\epsilon)>0$ such that
$$tr_{\omega_0}(\chi_0) \leq  |\sigma_{G_p}|^{-2\epsilon}_{h_{\Omega_0}}  |\sigma_{\mathcal{F}}|^{-2\epsilon'} _{h_{\Omega_0}}e^{C_4 + 3A\varphi_0}\leq C_5  |\sigma_{G_p}|^{-2\epsilon}_{h_{\Omega_0}}  |\sigma_{\mathcal{F}}|^{-2\epsilon'} _{h_{\Omega_0}}.$$
Since the constant $C_5$ does not depend on $\epsilon'$, the lemma is proved by letting $\epsilon'\rightarrow0$. 

\end{proof}

\begin{lemma} \label{fest} Let $X$ be a component of $\mathcal{X}_0$ and $p \in X \cap \mathcal{R}_{\mathcal{X}_0}$. For any $K \subset\subset \mathcal{X}_0 \setminus \textnormal{LCS} (\mathcal{X}_0)$, there exists $C_K>0$ such that for any $q\in X \cap \mathcal{R}_{\mathcal{X}_0}  \cap K$,    
\begin{equation}
d_{\mathbf{Y}}(p, q) \leq C_K.
\end{equation}

\end{lemma}

\begin{proof} We prove by contradiction.  Suppose there exist a sequence of points $q_j \in X \cap \mathcal{R}_{\mathcal{X}_0} \cap K$ such that 
$$d_{\mathbf{Y}}(q_j, p) \rightarrow \infty. $$
We can assume that $q_j \rightarrow q$ in  $(\mathcal{X}_0, \chi_0)$, where $q$ must be a singular point in  $\mathcal{X}_0$ away from $\textnormal{LCS} (\mathcal{X}_0)$.

Then there exists $r_0>0$ such that for all $j$, $B_{\chi_0}(q_j, r_0)$, the geodesic ball in $(\mathcal{X}_0, \chi_0)$ centered at $q_j$ with radius $r_0$,  lies outside an open neighborhood $U$ of $\textnormal{LCS}  (\mathcal{X}_0)$ in $(\mathcal{X}_0, \chi_0)$ since $q$ is away from $\textnormal{LCS} (\mathcal{X}_0)$. By  Lemma \ref{sch4} and geodesic convexity of $\mathcal{R}_{\mathbf{Y}}=\mathcal{R}_{\mathcal{X}_0}$ in $(\mathbf{Y}, d_{\mathbf{Y}})$, there exists $r_0'>0$ such that for all $j$, %
$$B_{d_{\mathbf{Y}}}(q_j, 2r_0') \cap \mathcal{R}_{\mathbf{Y}} \subset B_\chi(q_j, r_0)$$ and so 
$$B_{d_{\mathbf{Y}}}(q_j, 2r_0')\cap (U\cap \mathcal{R}_{\mathcal{X}_0})  = \phi . $$
%
%
Since the total volume of $(\mathbf{Y}, d_{\mathbf{Y}})$ is bounded and $d_{\mathbf{Y}}(q_j, p) \rightarrow \infty$, %
$$\epsilon_j = \textnormal{Vol}(B_{d_{\mathbf{Y}}}(q_j, 2r_0'), d_{\mathbf{Y}}) \rightarrow 0. $$

Now we construct the auxiliary smooth function $1\leq F_j \leq A_j$ on $\mathcal{X}_0$ such that 
$$F_j= 1, ~  \textnormal{on} ~ \mathcal{X}_0 \setminus (B_{d_{\mathbf{Y}}}(q_j, 2r_0') \cap \mathcal{R}_{\mathcal{X}_0} ) $$
 and
  $$F_j = A_j, ~\textnormal{on} ~ B_{d_{\mathbf{Y}}}(q_j, r_0')\cap \mathcal{R}_{\mathcal{X}_0}.$$
 Then we calculate the $L^{1+\delta}$-norm of $F_j$ on $\mathcal{X}_0$ for some fixed small $\delta>0$. Let $\omega_0$ be the unique K\"ahler-Einstein current on $\mathcal{X}_0$ with $\omega_0^n = e^{\varphi_0} \Omega_0$. There exists $C_1>0$ such that for all $j$, 
\begin{eqnarray*}
&&\int_{\mathcal{X}_0\setminus U} F_j^{1+\delta} \Omega_0 \\
&\leq&  \int_{\mathcal{X}_0\setminus \left( U \cup (B_{d_\mathbf{Y}}(q_j, 2r_0') \cap \mathcal{R}_{\mathcal{X}_0} )\right)} \Omega_0+ A_j^{1+\delta} e^{\sup_{B_{d_\mathbf{Y}}(q, 2r_0')\cap \mathcal{R}_{\mathcal{X}_0}}|\varphi_0|}  \int_{B_{d_\mathbf{Y}}(q_j, 2r_0')} dV_{d_{\mathbf{Y}}} \\
&\leq& C_1 + C_1A_j^{1+\delta} \epsilon_j
\end{eqnarray*}
We choose
$$A_j  = (\epsilon_j)^{- \frac{1}{1+\delta}} \rightarrow \infty$$ and so  
$$\int_{\mathcal{X}_0 \setminus U} (F_j)^{1+\delta} \Omega \leq 2C_1 $$
for some fixed $\delta>0$ and for all $j$.  

Now we consider the family of complex Monge-Amp\`ere equations on $\mathcal{X}_0$
\begin{equation}\label{dmaq}
(\chi_0 + \ddbar \phi_j) ^n = e^{\phi_j} F_j \Omega_0. 
\end{equation} 
By Lemma \ref{Finfty}, there exists a solution $\phi_j$ solving (\ref{dmaq}). Furthermore,  $\phi_j \in C^\infty(\mathcal{R}_{\mathcal{X}_0})$ and for any $K \subset \subset \mathcal{X}_0 \setminus U$, there exists $C_2>0$ such that for all $j$, 
$$ \sup_{ B_{d_{\mathbf{Y}}}(q_j, r_0')\cap \mathcal{R}_{\mathcal{X}_0}} |\phi_j|\leq C_2. $$ 
From equation (\ref{dmaq}),  $\chi_0+ \ddbar \phi_j$ is K\"ahler-Einstein on $B_{d_{\mathbf{Y}}}(q_j, r_0')\cap\mathcal{R}_{\mathcal{X}_0}$.

We will now prove a local Schwarz lemma. 
%
%
Let $\omega_j = \chi_0 + \ddbar \phi_j$ and let $\omega_0= \chi_0 + \ddbar \varphi_0$ be unique K\"ahler-Einstein current on $\mathcal{X}_0$. Let 
$$r(x)=d_{\mathbf{Y}}(x, p) $$ be the distance function from $x$ to $p$ in $(\mathbf{Y}, d_{\mathbf{Y}})$. By the geodesic convexity result in \cite{CN}, for any $x\in \mathcal{R}_{\mathcal{X}_0}\cap X$, there exists a smooth geodesic $\gamma$ joining $p$ and $x$, and $\gamma$ lies entirely in $\mathcal{R}_{\mathcal{X}_0}$. In particular, $r(x)$ is smooth on $\mathcal{R}_{\mathcal{X}_0}\cap X$ except at cut-locus points on smooth minimal geodesics.

We pick an effective divisor $\mathcal{F}$ similarly as in Lemma \ref{fest} such that $\mathcal{F} $ is numerically equivalent to $K_{\mathcal{X}_0}$ and it contains the singular locus $\mathcal{S}_{\mathcal{X}_0}$ of $\mathcal{X}_0$. We consider the quantity
$$H_{j, \varepsilon}= \log  \left( \psi  \frac{\omega_j^n}{\omega_0^n} \right) - \varepsilon (  \varphi_0 - \gamma \log |\sigma_{\mathcal{F}}|^2_{h_{\Omega_0}}),$$
where $\psi=\phi(r(x))$ is chosen with a smooth cut-off function $\phi$ satisfing
$$ \phi(r) = 1, ~ \textnormal{if} ~ r\leq r'_0, ~ \phi(r) =0, ~ \textnormal{if} ~ r\geq 2r'_0$$
and
$$\phi\geq 0, ~ 0 \leq \phi^{-2+\frac{1}{n} } (\phi')^2 \leq C, ~ \phi^{-1 + \frac{1}{n}} |\phi''| \leq C_3$$
for some fixed constant $C_3>0$. Since $\varphi_0$ is milder than any log pole singularities, $\varphi_0 - \gamma \log |\sigma_\mathcal{D}|^2_{h_\mathcal{D}}$ is uniformly bounded below after fixing $\gamma>0$. Immediately, we can conclude that 
$$\sup_{\mathcal{X}_0} H_{j, \varepsilon} < \infty $$ and $H_{j, \varepsilon}$ tends to $-\infty$ near $G_p$.  Suppose the maximum of $H_{j, \epsilon}$ is achieved at $q'$. Then $q'\in \left(  \mathcal{R}_{\mathcal{X}_0} \setminus Supp(G_p) \right) \cap Supp(\psi)$.   

If $q'$ is not a cut-point of $p$, then $H_{j, \varepsilon}$ is smooth near $q'$ and there exists $C_4>0$ such that
\begin{eqnarray*}
 \Delta_{\omega_0} H_{j, \varepsilon} &\geq&  tr_{\omega_0}(\omega_j) + \varepsilon (1-\gamma) tr_{\omega_0}( \chi_0  ) -n - n \varepsilon - \left(  \psi^{-2} |\nabla \psi|^2_{\omega_0} - \psi^{-1} \Delta_{\omega_0} \psi \right) \\
 &\geq& \left( \frac{ (\omega_j)^n}{ (\omega_0)^n} \right)^{1/n} - C_4 - C_4\psi^{-1/n}  .
 \end{eqnarray*}
Therefore at the maximum point $q'$ of $H_{j, \varepsilon}$, $\psi \left( \frac{ (\omega_j)^n}{ (\omega_0)^n} \right) \leq  (2C_4)^n$ and so there exists $C_5>0$ such that for all $j$ and $0 <\varepsilon <1$, 
$$H_{j, \varepsilon} \leq C_5. $$

 If $q$ is a cut-point of $p$, one can use the following trick of Calabi (c.f. \cite{CY}). Let $\gamma$ be a smooth minimizing geodesic joining $p$ and $q$ with $\gamma(0)=p$, $\gamma(r(q))=q$. Let $p_\delta= \gamma(\delta)$ for sufficiently small $\delta>0$. Obviously $q$ is not a cut-point of $p_\delta$. Let $r_\delta(x)$ be the distance function from $x$ to $p_\epsilon$ and $\psi_\delta(x) = \phi(r_\delta(x)+\delta)$. Since $r_\delta(x) +\delta=r_\delta(x)+r_\delta(p) \geq r_p(x)$, $\psi_\epsilon(x) \leq \psi(x)$ and $\psi_\delta(q)=\psi(q)$. One now can apply the maximum principle to 
 $$H_{j, \epsilon, \delta} = \log  \left( \psi_\delta \frac{(\omega_j)^n}{ (\omega_0)^n} \right) - \epsilon (  \varphi_0 - \gamma \log |\sigma_\mathcal{F}|^2_{h_{\Omega_0}})$$ 
  since the maximum of $H_{\epsilon, \delta}$ is also achieved at $q$.  We obtain the same estimate as in the non cut-locus case.

Finally we let $\epsilon \rightarrow 0$. Then there exists $C_6>0$ such that  for all $j$, we have
$$\sup_{B_{d_\mathbf{Y}}(q_j, r_0')\cap \mathcal{R}_{\mathcal{X}_0}} \frac{ (\omega_j)^n}{ (\omega_0)^n}  \leq C_6. $$
But from the equations for $\varphi_0$ and $\varphi_j$, we have 
$$\sup_{B_{d_\mathbf{Y}}(q_j, r_0')\cap \mathcal{R}_{\mathcal{X}_0}} \left( \frac{ (\omega_j)^n}{ (\omega_0)^n} \right) = \sup_{B_{d_\mathbf{Y}}(q_j, r_0')\cap \mathcal{R}_{\mathcal{X}_0}} F_j e^{\phi_j - \varphi_0}  =A_j \sup_{B_{d_\mathbf{Y}}(q_j, r_0')\cap \mathcal{R}_{\mathcal{X}_0}} e^{\phi_j - \varphi_0} \rightarrow \infty$$
as $\varphi_0$ and $\phi_j$ are uniformly bounded on $B_{d_\mathbf{Y}}(q_j, 1)\cap \mathcal{R}_{\mathcal{X}_0} $. Contradiction. 

\end{proof}

\begin{lemma} \label{infest} For any $q_j \in \mathcal{R}_{\mathcal{X}_0}$ converging to some $q\in \textnormal{LCS} (\mathcal{X}_0)$ in $(\mathcal{X}_0, \chi_0)$, we have
\begin{equation}
\lim_{j\rightarrow \infty} d_{\mathbf{Y}}(p, q_j) = \infty. 
\end{equation}

\end{lemma}

\begin{proof}  We prove by contradiction. Suppose there exist a sequence of points $q_j \in \mathcal{R}_{\mathcal{X}_0} \rightarrow q \in \textnormal{LCS} (\mathcal{X}_0)$ with respect to $\chi_0$ and $d_{\mathbf{Y}}(q_j, p) \leq D$ for some fixed $D>0$. We let $q' = \lim_{j \rightarrow \infty} q_j$ with respect to $d_{\mathbf{Y}}$. Then using the partial $C^0$-estimate in Corollary \ref{par1}, there exists a global $L^2$-section $\sigma$ on $\mathbf{Y}$ such that $\sigma$ is a global section of $mK_{\mathbf{Y}}$ for some $m$ and $\sigma$ does not vanish near $q'$. In particular, $|\sigma|^2_{h_\mathbf{Y}}$ is bounded above and below away from $0$ near $q'$ in $\mathbf{Y}$, where $h_\mathbf{Y}$ is the hermitian metric on $K_\mathbf{Y}$ induced from the K\"ahler-Einstein volume form on $\mathcal{R}_\mathbf{Y}$.  Hence we can write $h_\mathbf{Y} = e^{-\psi}$ for some bounded plurisubharmonic function near $q'$ because $\mathbf{Y}$ is an analytic normal space and $-\ddbar \log |\sigma|^2_{h_\mathbf{Y}}= \ddbar \psi$ is the K\"ahler-Einstein metric.

On the other hand, $\Omega_0 = \left( \sum_{i=0}^N \eta_i \wedge \overline{ \eta_i} \right)^{1/m}$ and the K\"ahler-Einstein volume form is given by $e^{\varphi_0} \Omega_0$.  Each $\eta_i$ is a holomorphic section of $mK_{\mathbf{Y}}$ near $q'$ on the regular part of $\mathbf{Y}$. Since $\mathbf{Y}$ is normal near $q'$ by the result of \cite{DS2} as a local version of \cite{DS1}, $\eta_i$ extends to a holomorphic section over $q'$ and therefore near $q'$. Compare the two volume measures, there exists $C>0$ such that near $q'$ in $(\mathbf{Y}, d_{\mathbf{Y}})$, 
$$\varphi_0= \psi + \frac{1}{m} \log \left( \frac{|\sigma|^2}{\sum_{i=0}^N (\eta_i \wedge \overline{\eta_i} )} \right) \geq - C.$$
On the other hand, $\varphi_0(q_j) \rightarrow - \infty.$
Contradiction. 


\end{proof}

\begin{corollary} \label{podis} For any $p \in \mathcal{R}_{\mathbf{Y}}$ and    any $R\geq 1$, there exists $C= C(p , R)>0$ such that  
\begin{equation}
\sup_{B_{d_\mathbf{Y}}(p , R) \cap \mathcal{R}_{\mathcal{X}_0} } |\varphi_0 | \leq C. 
\end{equation}

\end{corollary}

\begin{proof}   We prove by contradiction. Suppose there exist a sequence of points $q_j \in B_{d_\mathbf{Y} }(p, R) \cap \mathcal{R}_{\mathcal{X}_0} \subset \mathbf{Y} $ such that $\varphi_0(q_j) \rightarrow -\infty$ as $j \rightarrow \infty$. Then after passing to a subsequence, we can assume that $q_j$ converges to $q' \in \textnormal{LCS}(\mathcal{X}_0)$ in $(\mathcal{X}_0, \chi_0)$ and  $q_j$ converges to some $q\in B_{d_\mathbf{Y}}(p_0, 2R)$ in $(\mathbf{Y}, d_\mathbf{Y})$. This leads to contradiction because $d_\mathbf{Y}(p, q_j) \rightarrow \infty$ from Lemma \ref{infest}. 

\end{proof}

In conclusion, in each component of $(\mathbf{Y}, d_\mathbf{Y})$, the local boundedness of the K\"ahler-Einstein potential is equivalent to the boundedness of distance in a uniform way.


\section{Proof of Theorem \ref{main2} }

In this section, we will complete the proof of Theorem \ref{main2}.  We first derive a geometric Schwarz lemma.

\begin{lemma} \label{sch7} For any $p \in \mathcal{R}_{\mathcal{X}_0}$ and  any $R\geq 1$, there exists $C= C( R)>0$ such that   
\begin{equation}\label{compsch1}
\omega_0 \geq C ~ \chi_0, 
\end{equation}
on $B_{d_\mathbf{Y}}(p , R)\cap R_{\mathbf{Y}}$.

\end{lemma}

\begin{proof} We let $r(x)$ be the distance function from $x\in \mathbf{Y}$ to $p $ in $(\mathbf{Y}, d_\mathbf{Y})$. 
We choose the smooth cut-off function $\psi=\phi(r(x))$ satisfying
$$ \phi(r) = 1, ~ if ~ r\leq R, ~ \phi(r) =0, ~ if ~ r\geq 2R.$$
Furthermore, we can assume that
$$\phi\geq 0, ~ 0 \leq \phi^{-1} (\phi')^2 \leq C R^{-2}, ~ |\phi''| \leq C R^{-2}. $$
%
%

We consider the quantity $ tr_{g_0}(\chi_0)$ on $\mathcal{R}_\mathbf{Y}=\mathcal{R}_{\mathcal{X}_0}$.
Then straightforward calculations show that there exists $K>0$ such that on $\mathcal{R}_{\mathcal{X}_0}$, we have 
$$ \Delta_{\omega_0}  \log tr_{\omega_0}(\chi_0) \geq -K (tr_{\omega_0} (\chi_0)) - K. $$
Now we choose the same effective divisor $\mathcal{F}$ as in the proof of Lemma \ref{fest} such that $\mathcal{F}$ is numerically equivalent to $K_{\mathcal{X}_0}$ and it contains the singular locus $\mathcal{S}_{\mathcal{X}_0}$ of $\mathcal{X}_0$. Let $\sigma_{\mathcal{F}}$ be the defining function of $\mathcal{F}$. Then we consider the following quantity
$$H_\epsilon =   \log  \left( \psi tr_{\omega_0}(\chi_0)   \right) -  A \varphi_0   +\epsilon \log |\sigma_\mathcal{F}|^2_{h_{\Omega_0}}.$$
There exists $C_1 >0$ such that  
\begin{eqnarray*}
&& \Delta_{\omega_0} H_\epsilon \\
&=& \Delta_{\omega_0}  ( \log tr_{\omega_0} (\chi_0) -  A\varphi_0) -\epsilon tr_{\omega_0}(\chi_0)- \left(  \psi^{-2} |\nabla \psi|^2 - \psi^{-1} \Delta_{g_t} \psi \right) \\
&\geq& tr_{\omega_0}(\chi_0)   - An  - C_1 \psi^{-1}, 
\end{eqnarray*}
by  fixing a sufficiently large $A$. $H_\epsilon$ is smooth on $\mathcal{R}_{\mathcal{X}_0} \setminus \mathcal{F}$ and by Lemma \ref{sch1},  $H_\epsilon$ tends to $-\infty$ near $\mathcal{F}$ for any fixed $\epsilon>0$. If $H_\epsilon$ achieves its positive maximum at $z_0$ in the interior of $B_{d_{\mathbf{Y}}}(p, 2R)$, then there exists $C_2>0$ such that
$$H_\epsilon(z_0) \leq    C_2 \sup_{B_{\mathbf{Y}}(p, R) \cap (\mathcal{R}_{\mathcal{X}_0}\setminus \mathcal{F})} |\varphi_0| $$
Since $ \sup_{B_{\mathbf{Y}}(p, R) \cap (\mathcal{R}_{\mathcal{X}_0}\setminus \mathcal{F})} |\varphi_0|  $ is uniformly bounded by Corollary \ref{podis}, there exists $C_3 >0$ such that 
$$  \sup_{B_{\mathbf{Y}}(p, R) \cap (\mathcal{R}_{\mathcal{X}_0}\setminus \mathcal{F})} tr_{\omega_0}(\chi_0)  \leq  C_3 . $$
This proves the estimate (\ref{compsch1}).

\end{proof}

\begin{corollary} There exists a unique map $$\Phi: (\mathbf{Y}, d_{\mathbf{Y}}) \rightarrow (\mathcal{X}_0, \chi)$$ extending the identity map from $\mathcal{R}_\mathbf{Y}$ and $\mathcal{R}_{\mathcal{X}_0}$. Furthermore, $\Phi$  is a Lipschitz map and  $$\Phi(\mathbf{Y}) = \mathcal{X}_0\setminus \textnormal{LCS} (\mathcal{X}_0) .$$ 

\end{corollary}

\begin{proof} For any point $p$ in $\mathcal{R}_{\mathbf{Y}}$ and  for any $R>0$, by Lemma \ref{sch7}, there exists $C_R>0$ such that  
$$\omega_0 \geq C_R~ \chi_0$$
on $B_{d_\mathbf{Y}}(p, R )\cap \mathcal{R}_{\mathcal{X}_0}$.  The geodesic convexity of $\mathcal{R}_{\mathcal{X}_0}$ in $\mathbf{Y}$  implies that the identity map uniquely extends from $\mathcal{R}_{\mathbf{Y}}$ to $\mathbf{Y}$ and $\Phi$ is locally Lipschitz.  Therefore any Cauchy sequence of points in $(\mathbf{Y}, d_\mathbf{Y}) $ must also be a Cauchy sequence in $(\mathcal{X}_0, \chi_0)$. The last statement of the corollary follows from  the distance estimate in  Lemma \ref{fest} and Lemma \ref{infest}.

\end{proof}

%

\begin{lemma}

$\Phi$ is injective. 

\end{lemma}

\begin{proof} For any two distinct points $q_1, q_2 \in \mathbf{Y}$, the partial $C^0$-estimates in Corollary \ref{par1} allow us to construct two  holomophic sections of $mK_{\mathbf{Y}}$ on $\mathbf{Y}$ for some sufficiently large $m$ such that 
$$|\sigma_1(q_1)|^2_{(h_\mathbf{Y})^m} \geq  2|\sigma_2|^2_{(h_\mathbf{Y})^m}(q_1)\neq 0, ~~ |\sigma_2(q_2)|^2_{(h_\mathbf{Y})^m} \geq  2|\sigma_1|^2_{(h_\mathbf{Y})^m}(q_2) \neq 0$$
and $$\int_{\mathbf{Y}} |\sigma_i|^2_{(h_\mathbf{Y})^m} dV_{d_{\mathbf{Y}}} < \infty, ~ i=1, 2.$$
Both $\sigma_1$ and $\sigma_2$ extend uniquely to pluricanonical sections $\sigma'_1$ and $\sigma'_2$ on $\mathcal{X}_0$. Then we consider the meromorphic function 
$$f = \frac{\sigma_2}{\sigma_1} ~\textnormal{on} ~ \mathbf{Y},  ~~~  f' =\frac{\sigma'_2}{\sigma'_1}  ~\textnormal{on} ~\mathcal{X}_0 . $$
Then $f = \Phi^* (f')= f'(\Phi)$. $f$ and $f'$ are holomorphic near $q_1$, $q_2$ and $\Phi(q_1)$, $\Phi(q_2)$ respectively and 
$$ f(q_1) \leq 1/2, ~ f(q_2) \geq 2, $$ Therefore $f'(\Phi(q_1)) \neq f'(\Phi(q_2))$ and so $$ \Phi(q_1) \neq \Phi(q_2).$$

\end{proof}

\begin{corollary} $\mathbf{Y}$ is homeomorphic to $\mathcal{X}_0\setminus \textnormal{LCS} (\mathcal{X}_0). $

\end{corollary}

\begin{proof} It suffices to show that $\Phi^{-1}: \mathcal{X}_0\setminus \textnormal{LCS}(\mathcal{X}_0) \rightarrow \mathbf{Y}$ is continuous. Let $\{ x_j\}_{j=1}^\infty$ a sequence of points in a fixed component of $\mathcal{X}_0\setminus \textnormal{LCS}(\mathcal{X}_0) $ such that $x_j \rightarrow x_\infty \in \mathcal{X}_0\setminus \textnormal{LCS}(\mathcal{X}_0) $ with respect to $\chi_0$. By Lemma \ref{fest}, there exists a base point $p \in \mathcal{R}_\mathbf{Y}$ and $C>0$ such that for all $j$, 
$$d_\mathbf{Y}(p, x_j) \leq C. $$
Let $y_{j_l}$ be any convergent subsequence of $y_j = \Phi^{-1}(x_j)$ in $(\mathbf{Y}, d_\mathbf{Y})$ and let $y_\infty$ be the limit point. Then by continuity of $\Phi$, we have
$$ x_\infty = \lim_{l_j \rightarrow \infty} x_{lj} = \lim_{l_j \rightarrow \infty} \Phi(y_{l_j}) = \Phi(y_\infty). $$
Therefore $y_\infty$ is uniquely determined and $y_\infty = \Phi^{-1}(x_\infty)$.

\end{proof}

We have now completed the proof of Theorem \ref{main2} by combining all the previous results we obtained. Theorem \ref{main3} immediately follows from Theorem \ref{main2} by using the algebraic result of Hacon and Xu \cite{HX}.

\bigskip

\noindent {\bf{Acknowledgements:}} The author would like to thank Xiaowei Wang for many stimulating conversations.

\end{document}